\documentclass[12pt,reqno]{amsart}
\usepackage{url}
\usepackage[utf8]{inputenc}
\usepackage{mathabx} 
\usepackage{bbm}     
\usepackage{lscape}
\usepackage[all]{xy}
\usepackage{xcolor}
\usepackage{cases}
\usepackage{stmaryrd}

\usepackage{hyperref}
\hypersetup{
    colorlinks,
    linkcolor={blue},
    citecolor={blue},
    urlcolor={blue!80!black}
}

\newtheorem{Thm}{Theorem}[section]
\newtheorem{Lem}[Thm]{Lemma}
\newtheorem{Cor}[Thm]{Corollary}
\newtheorem{Prop}[Thm]{Proposition}

\theoremstyle{definition}
 \newtheorem{Rem}[Thm]{Remark}

\newcommand{\blue}[1]{\textcolor{blue}{#1}}

\newcommand{\FF}{\mathbb{F}}
\newcommand{\GG}{\mathbb{G}}
\newcommand{\NN}{\mathbb{N}}

\newcommand{\RR}{\mathbb{R}}
\newcommand{\ZZ}{\mathbb{Z}}
\newcommand{\Z}{\mathbb{Z}}   

\newcommand{\cT}{{\mathcal T}}

\newcommand{\Del}{\Delta}

\newcommand{\Ome}{\Omega}

\newcommand{\alp}{\alpha}
\newcommand{\bet}{\beta}

\newcommand{\eps}{\varepsilon}
\newcommand{\veps}{\epsilon}

\newcommand{\sig}{\sigma}
\newcommand{\vph}{\varphi}

\newcommand{\df}{\colon}

\newcommand{\diag}{\operatorname{diag}}

\newcommand{\Red}{\operatorname{Red}}
\newcommand{\Loc}{\operatorname{Loc}}

\newcommand{\rk}{\operatorname{rank}}

\newcommand{\rep}{\operatorname{rep}}

\newcommand{\add}{\operatorname{add}}

\newcommand{\gldim}{\operatorname{gl.dim}}
\newcommand{\pdim}{\operatorname{proj.dim}}
\newcommand{\idim}{\operatorname{inj.dim}}

\newcommand{\dimv}{\underline{\dim}}
\newcommand{\rkv}{\underline{\rk}}

\newcommand{\rad}{\operatorname{rad}}

\newcommand{\fac}{\operatorname{fac}}

\newcommand{\Hom}{\operatorname{Hom}}
\newcommand{\Ext}{\operatorname{Ext}}
\newcommand{\End}{\operatorname{End}}
\newcommand{\Aut}{\operatorname{Aut}}

\newcommand{\Ker}{\operatorname{Ker}}

\newcommand{\GL}{\operatorname{GL}}

\newcommand{\Quot}{\operatorname{Frac}}

\newcommand{\bil}[1]{\langle #1\rangle}
\newcommand{\abs}[1]{\left| #1\right|}

\newcommand{\Iff}{\Longleftrightarrow}
\newcommand{\ra}{\rightarrow}
\newcommand{\lcm}{\operatorname{lcm}}

\newcommand{\real}{\mathrm{re}}

\newcommand{\bbo}{\mathbbm{1}}

\newcommand{\lf}{{\mathrm{lf}}}        
\newcommand{\rS}{{\mathrm{rS}}}        
\newcommand{\Mat}{\operatorname{Mat}}  
\newcommand{\NC}{\operatorname{NC}}  
\newcommand{\ualp}{\underline{\alp}}
\newcommand{\tS}{{\widetilde{S}}}
\newcommand{\tH}{{\widetilde{H}}}
\newcommand{\hS}{{\widehat{S}}}
\newcommand{\hH}{{\widehat{H}}}
\newcommand{\hB}{{\widehat{B}}}
\newcommand{\pow}[1]{[\![ #1]\!]} 
\newcommand{\lpw}[1]{(\!( #1 )\!)} 
\newcommand{\ul}[1]{{\underline{#1}}}

\newcommand{\bsm}{\left(\begin{smallmatrix}}
\newcommand{\esm}{\end{smallmatrix}\right)}

\newcommand{\bbsm}{\left[\begin{smallmatrix}}
\newcommand{\besm}{\end{smallmatrix}\right]}

\newcommand{\bit}{\begin{itemize}\itemsep2mm}
\newcommand{\eit}{\end{itemize}}

\begin{document}
\date{29.06.2019}
\parskip10pt

\title{Rigid modules and Schur roots}

\author{Christof Gei{\ss}}
\address{Christof Gei{\ss}, Instituto de Matem\'aticas,
Universidad Nacional Aut\'onoma de M\'exico,
Ciudad Universitaria,
04510 Cd.~de M\'exico,
MEXICO}
\email{christof.geiss@im.unam.mx}

\author{Bernard Leclerc}
\address{Bernard Leclerc,
LMNO, Univ. de Caen,
CNRS, UMR 6139,
F-14032 Caen Cedex,
FRANCE}
\email{bernard.leclerc@unicaen.fr}

\author{Jan Schr\"oer}
\address{Jan Schr\"oer,
Mathematisches Institut,
Universität Bonn,
Endenicher Allee 60,
53115 Bonn,
GERMANY}
\email{schroer@math.uni-bonn.de}


\begin{abstract}
Let $C$ be a symmetrizable generalized Cartan matrix with 
symmetri-zer 
$D$ and orientation $\Ome$. 
In \cite{GLS1} we constructed for any field $\FF$ an $\FF$-algebra 
$H := H_\FF(C,D,\Ome)$, defined in terms of a quiver with relations, such that the locally free $H$-modules behave in many aspects like representations of a hereditary algebra $\tH$ of the corresponding type. 
We define
a Noetherian algebra $\hH$ over a power series ring, which provides
a direct link between the representation theory of $H$ and
of $\tH$.
We define and study a reduction and a localization functor
relating the module categories of $\hH$, $\tH$ and $H$.
These are used to show that there are natural bijections between the
sets of isoclasses of
tilting modules over the three algebras $\hH$, $\tH$ and $H$.
We show that the indecomposable rigid locally free 
$H$-modules are parametrized, via their rank vectors, by the real Schur roots associated to $(C,\Ome)$. 
Moreover, the left finite bricks of $H$, in the sense of Asai, are parametrized, via their dimension vectors, by the real Schur roots associated to $(C^T,\Ome)$.
\end{abstract}

\maketitle
\setcounter{tocdepth}{1}
\tableofcontents


\section{Introduction}\label{sec:intro}


\subsection{Main results}\label{subsec:mainresults}
Let $I=\{1, 2, \ldots, n\}$ and $C\in\ZZ^{I\times I}$ be a generalized
symmetrizable Cartan matrix with 
symmetrizer $D=\diag((c_i)_{i\in I})$ for some $(c_i)_{i\in I}\in\ZZ_{>0}^I$. Let $\Ome\subset I\times I$ be an orientation of $C$ (see \cite[Section 1.4]{GLS1}). 
Without loss of generality we
may assume that $(i,j)\in\Ome$ implies $i<j$ for the natural ordering of~$I$.

Let $\FF$ be a field. In \cite{GLS1} we introduced a finite-dimensional,
1-Iwanaga-Gorenstein algebra $H:=H_\FF(C,D,\Ome)$ in terms of a quiver with
relations, such that the exact category $\rep_\lf(H)$ of locally free $H$-modules
resembles in many aspects the representation theory of a 
finite-dimensional
hereditary algebra of type $(C,\Ome)$. 
The algebras $H$ (and their associated generalized preprojective
algebras, which were also introduced in \cite{GLS1}) provide
a new framework relating the representation theory of symmetrizable Kac-Moody algebras with the representation theory of quivers with relations.
(Such a framework was previously only available for the symmetric case.)

We consider
the Noetherian $\FF\pow{\veps}$-algebra 
\[
\hH:=\varprojlim_k H_\FF(C,kD,\Ome)
\]
together with its canonical homomorphism to $H$ and to the localization
$\tH = \hH_\veps$.
Here $\veps$ is a certain central element of $\hH$, and $\tH$ is a species over the field $\FF\lpw{\veps}$.
In particular, $\tH$ is a finite-dimensional hereditary 
$\FF\lpw{\veps}$-algebra.
The algebras $\hH$, $H$ and $\tH$ are related via the
\emph{reduction functor}
\[
\Red := H \otimes_\hH -\df \rep_\lf(\hH) \to \rep_\lf(H)
\]
and the \emph{localization functor}
\[
\Loc := \tH \otimes_\hH -\df \rep_\lf(\hH) \to \rep(\tH).
\]

It turns out that $\hH$ is an $\FF\pow{\veps}$-order in $\tH$, and the 
$\hH$-lattices are exactly the locally free $\hH$-modules.
The representation theory of orders plays a central role in the
representation theory of finite groups, see for example 
Curtis and Reiner's book \cite{CR81}.
Note that unlike the classical situation, our \emph{ambient algebra} $\tH$
is hereditary and not semisimple. 
Let us also remark that the global dimension of $\hH$ is (with 
the exception of some trivial cases) equal to 2.

The following is our first main result.

\begin{Thm}\label{thm:mainresult1}
The functors
\[
\xymatrix@+3ex{
\rep_\lf(\hH) 
\ar[r]^{\Loc} 
\ar[d]^{\Red} & \rep(\tH)
\\
\rep_\lf(H)
}
\]
induce bijections
\[
\xymatrix{
\{ \text{rigid locally free $\hH$-modules} \}/\!\!\cong \;\;
\ar[r]\ar[d] & \;\;
\{ \text{rigid $\tH$-modules} \}/\!\!\cong 
\\
\{ \text{rigid locally free $H$-modules} \}/\!\!\cong .
}
\]
These bijections and their inverses preserve indecomposability.
Furthermore, they preserve tilting modules and induce isomorphisms
\[
\xymatrix{
\cT(\hH)
\ar[r]\ar[d] & \;\;
\cT(\tH)
\\
\cT(H)
}
\]
between the exchange graphs of support tilting pairs for the
algebras $\hH$, $\tH$ and $H$.
\end{Thm}

By \cite{CK06} for the symmetric and \cite{Ru15} (which is based on \cite{Hu}) for the symmetrizable case, the exchange graph $\cT(\tH)$ is a categorical
realization of the exchange graph of an acyclic cluster algebra
of type $(C,\Omega)$ (here $(C,\Omega)$ encodes an acyclic valued quiver).
Thus Theorem~\ref{thm:mainresult1} provides a new
class of such categorical realizations.

For the proof of Theorem~\ref{thm:mainresult1} we use 
Demonet's Lemma~\ref{lem:demonet} which states that $\tau$-tilting $H$-modules are actually classical
(in particular locally free) tilting modules. 
Our result is similar to Crawley-Boevey's classification of the rigid integral representations of a quiver \cite{CB96}. However, the approach of \cite{CB96} is based on exceptional sequences and 
does not seem to work in our case.

For $H = H_\FF(C,D,\Ome)$ as above, the homological
bilinear form descends to the Grothendieck group $K_0(\rep_\lf(H))\cong\ZZ^I$,
see \cite[Section~4]{GLS1}, giving it the structure of a generalized Cartan
lattice 
\[
L(C,D,\Ome):=\left(\ZZ^I,\bil{-,-},\ualp\right)
\]
(in the sense of \cite{HuK16})
with orthogonal exceptional sequence
$\ualp=(\alp_1,\alp_2,\ldots,\alp_n)$.
Here, $(\alp_i)_{i\in I}$ is the standard
coordinate basis of $\ZZ^I$ and $\alp_i$ corresponds to the class of the
generalized simple $E_i\in\rep_\lf(H)$. 
The corresponding simple $H$-module is denoted by $S_i$.
We refer to \cite{GLS1} for the definition of $E_i$.
In this context we write $\rkv_H(M)$
for the class of a locally free $H$-module $M$ in the above Grothendieck group.
It is easy to see that 
each generalized Cartan lattice is isomorphic to some $L(C,D,\Ome)$.

The generalized Cartan lattice $L(C,D,\Ome)$ comes with its Weyl group 
\[W=W(C)<\Aut(\ZZ^I),\]
generated by the simple reflections
$s_i=s_{\alp_i}$ with $i\in I$, and a Coxeter element $s_1s_2\cdots s_n\in W$ compatible with the orientation $\Omega$. We can thus introduce the set of real roots
\[
\Del_\real=\Del_\real(C)
:= \bigcup_{i\in I} W\alpha_i
\subset \ZZ^I
\]
and the poset of non-crossing partitions
\[
  \NC(C,\Ome):=\{w\in W\mid 1\leq w\leq s_1s_2\cdots s_n\},
\]
where $\leq$ denotes the \emph{absolute order} on $W$ (see below, Section~\ref{sec:schurroots}). 
By a slight abuse of notation
we define the set of \emph{real Schur roots} as
\[
  \Del_\rS(C,\Ome):=\{\alp\in\Del_\real^+\mid s_\alp\in\NC(C,\Ome)\}.
\]
We can now state our second main result.

\begin{Thm}\label{thm:mainresult2}
For $H = H_\FF(C,D,\Ome)$ the following hold:
\bit

\item[(a)]
$M\mapsto \rkv_H(M)$ induces a bijection 
\[
\xymatrix{
\{ \text{indecomposable rigid locally free $H$-modules} \}/\!\!\cong 
\;\;\ar[r] &\;\; \Del_\rS(C,\Ome).  
}
\]

\item[(b)]
If $M\in\rep_\lf(H)$ is indecomposable rigid, then
\[
\End_H(M)\cong\FF[\veps]/(\veps^{c_i})
\] 
where $c_i=\bil{\rkv_H(M),\rkv_H(M)}$ for some  $i\in I$. 
Moreover, $M$ is free as an $\End_H(M)$-module.

\item[(c)]
$N\mapsto \dimv(N)$ induces a bijection 
\[
\xymatrix{
\{ \text{left finite bricks in $\rep(H)$} \}/\!\!\cong 
\;\;\ar[r] &\;\; \Del_\rS(C^T,\Ome).  
}
\]

\eit
\end{Thm}

If $A$ is a finite-dimensional hereditary algebra it is easy to see that the
Grothen-dieck group $K_0(A)$ of finite-dimensional $A$-modules equipped with
the Euler bilinear form, and an adequate ordering of the classes of the
simple modules form a generalized Cartan lattice together with an orthogonal
exceptional sequence of the above form. We say then, that $A$ is of type
$(C,\Ome)$. In this setup the class of a module
is identified with its dimension vector. By results of Crawley-Boevey \cite{CB93} and Ringel \cite{Rin94}, this correspondence induces a bijection between the isoclasses
of indecomposable rigid $A$-modules and the set $\Del_\rS(C,\Ome)$, see for
example \cite[Corollary~4.8]{HuK16}. However, in this situation the endomorphism
ring of any indecomposable rigid $A$-module is isomorphic to the endomorphism
ring of some simple $A$-module, which by Schur's lemma is a finite-dimensional division algebra.
This led to the name \emph{Schur root} for the above mentioned class of positive
real roots. 

Theorem~\ref{thm:mainresult1} is used for the proof
of Theorem~\ref{thm:mainresult2}(a),(b).
We will see in Section~\ref{subsec:pf-c} that  Theorem~\ref{thm:mainresult2}(c) is
an almost formal consequence of parts (a) and (b) in view
of \cite[Theorem~4.1]{DIJ17} and Demonet's Lemma~\ref{lem:demonet}.

In case $C$ is a Cartan matrix of finite type, Theorem~\ref{thm:mainresult2} has a relatively easy proof based
on results in \cite{GLS1}.
Namely, in the finite type case,
all positive roots are real Schur roots and part (a)
is \cite[Theorem~1.3]{GLS1}. Part (b) can be deduced from the results
in \cite{GLS1} by elementary Auslander-Reiten and tilting theory.

\subsection{Structure of this article}
Section~\ref{sec:schurroots} recalls some definitions and basic
facts on real Schur roots.
In
Section~\ref{sec:defH} we recall the definition and some
basic properties of the algebras
$H = H_\FF(C,D,\Ome)$.
The species $\tH$ and the Noetherian algebra $\hH$ 
are defined and studied in Section~\ref{sec:species}.
There we also explain that $\hH$ is an $\FF\pow{\veps}$-order
in $\tH$.
Section~\ref{sec:redloc} deals with a reduction functor
$\rep_\lf(\tH) \to \rep_\lf(H)$ and a localization functor
$\rep_\lf(\hH) \to \rep(\tH)$.
We show that rigid locally
free $H$-modules are up to isomorphism determined by their rank vectors.
The same section 
includes the proof of the vertical bijection in Theorem~\ref{thm:mainresult1}.
We also show that the horizontal map in Theorem~\ref{thm:mainresult1} is injective.
The main result of
Section~\ref{sec:tilting} are isomorphisms of the 
exchange graphs of tilting modules over the three 
algebras $H$, $\hH$ and $\tH$.
Section~\ref{subsec:proof2} finishes the proofs of 
Theorem~\ref{thm:mainresult1} and
Theorem~\ref{thm:mainresult2}(a).
Theorem~\ref{thm:mainresult2}(b) is proved in
Section~\ref{subsec:proof3}.
As an easy consequence of (a) and (b), Theorem~\ref{thm:mainresult2}(c) is shown in Section~\ref{subsec:pf-c}.
Finally,
Section~\ref{sec:examples} contains some examples.

\subsection{Conventions}
For an algebra $A$ let $\rep(A)$ denote the category
of finitely generated left $A$-modules.
If not indicated otherwise, by an $A$-\emph{module} we
mean a module in $\rep(A)$.
By a \emph{subcategory} of $\rep(A)$ we always mean a full
additive subcategory.
For $M \in \rep(A)$ let $\add(M)$ be the subcategory
of $\rep(A)$ whose objects are isomorphic to the direct summands of finite direct sums
of copies of $M$.
The category $\rep(A)$ has the 
\emph{Krull-Remak-Schmidt property} if each $M \in \rep(A)$ has
a direct sum decomposition
$M = M_1 \oplus \cdots \oplus M_t$
such that $\End_A(M_i)$ is local for all $i$.
As a consequence, the modules $M_i$
in this decomposition
are uniquely determined up to isomorphism and reordering.
Let $|M|$ be the number of 
isomorphism classes of indecomposable summands appearing in
such a direct sum decomposition.
Then $M$ is called \emph{basic} if $|M| = t$.
Finally, an $A$-module $M$ is \emph{rigid} if $\Ext_A^1(M,M) = 0$.


\section{Schur roots}
\label{sec:schurroots}
%
Following \cite[Section~4]{GLS1}, the bilinear form
\[
\bil{-,-}\df\ZZ^I\times\ZZ^I\ra\ZZ
\]
of the generalized Cartan lattice $L(C,D,\Ome)$ mentioned in
Section~\ref{subsec:mainresults} is defined by
\[
 \bil{\alp_i,\alp_j}= \begin{cases} 
   c_i        &\text{if } i=j,\\
   c_i c_{ij} &\text{if } i<j,\\
   0          &\text{otherwise}. \end{cases}
\]
Here $(\alpha_i)_{i\in I}$ denotes the standard basis of $\Z^n$.
We consider the symmetrization 
\[
(-,-)\df \RR^I \times \RR^I \to \RR
\]
of $\bil{-,-}$, which is defined by
\[
(\alp_i,\alp_j):=\bil{\alp_i,\alp_j}+\bil{\alp_j,\alp_i}=\alp_i^T D C\alp_j.
\]
For each $\bet\in\RR^I$ with $(\bet,\bet)\neq 0$ we define 
a reflection
$s_\bet\in\GL(\RR^I)$ by
\[
s_\bet(u) := u-\frac{2(\bet,u)}{(\bet,\bet)}\bet \qquad (u \in \RR^I).
\]
Observe that $s_\bet=s_{t\bet}$ for all $t\in\RR \setminus \{0\}$. 
Let us abbreviate
$s_i:=s_{\alp_i}$ $(i\in I)$ for the simple reflections, and observe that each $s_i$ induces an automorphism of the lattice $\ZZ^I$ since
\[
s_i(\alp_j)= \alp_j -c_{ij}\alp_i \qquad (i,j\in I).
\]
Then $W=W(C)$ is the subgroup of $\GL(\RR^I)$ which is generated by the
simple reflections, and the set of
\emph{real roots} is 
\[
  \Del_\real(C):=\bigcup_{i\in I} W\alp_i \subset \Z^I.
\]
With 
\[
c := \lcm((c_i)_{i\in I})
\] 
we observe that $D':=\diag((c/c_i)_{i\in I})$ is a symmetrizer of the transposed
Cartan matrix $C^T$.
To each real root $\beta = w(\alp_j)$ is associated the 
\emph{real coroot}
\[
\beta^\vee :=\frac{2\bet}{(\bet,\bet)} = w(\alp^\vee_j). 
\]
Note that each $w\in W$ defines an automorphism of the lattice spanned by the
simple coroots since
\[
s_i(\alp_j^\vee) = \alp_j^\vee - c_{ji}\alp_i^\vee \qquad (i,j\in I).
\]
Thus we have in particular 
\[
\bet^\vee \in \bigoplus_{i\in I} \ZZ\alp_i^\vee 
\quad\text{and}\quad
s_\bet = s_{\bet^\vee} = w^{-1} s_j w \in W.
\]
For $\beta \in \Delta_\real(C)$ we define the \emph{scaled real
coroot}
\[
\widetilde{\bet} := \sqrt{c} \beta^\vee.
\]
Then we have
\[
(\widetilde{\alp}_i,\widetilde{\alp}_j) = \frac{c}{c_i}c_{ji}.
\]
The scaled simple coroots 
$\widetilde{\alp}_i$ can be regarded  as the simple roots of the dual root system 
associated with the generalized Cartan matrix $C^T$. 
This allows us to identify
the set $\Delta_\real(C^T)$ of dual real roots  with the set
$\{ \widetilde{\bet} \mid \beta \in \Del_\real(C) \}$. 
Clearly, we have here
$W(C) = W(C^T)$ as a Coxeter group. 
For
\[
\beta = \sum_{i \in I} b_i \alpha_i\in \Delta_\real(C)\subset\ZZ^I,
\]
it is straightforward that
\[
\widetilde{\bet} = \sum_{i \in I} \frac{2c_ib_i}{(\beta,\beta)} \widetilde{\alpha}_i
\in\bigoplus_{i\in I}\ZZ\widetilde{\alp}_i.
\]

The \emph{absolute length} $l(w)$ of $w\in W$ is the minimal $r\geq 0$ such
that $w$ can be written as a product of reflections
\[
w=s_{\bet_1}s_{\bet_2}\cdots s_{\bet_r} \mbox{ with } \bet_i\in\Del_\real(C).
\]
The 
\emph{absolute order} on $W$ is defined as
\[
  u\leq v \Iff l(u)+l(u^{-1}v)=l(v).
\] 
Following \cite{HuK16}, we can now define the set of \emph{real Schur roots} as
\[
\Del_\rS(C,\Ome) := \{\beta \in \Del_\real(C) \mid s_\beta \le s_1s_2\cdots s_n\}. 
\]

By the above discussion we can identify the Cartan lattice $L(C^T,D',\Ome)$
with $\oplus_{i\in I}\ZZ\widetilde{\alp}_i$, and 
the map $\beta \mapsto \widetilde{\bet}$ is a bijection 
$\Del_\real(C) \to \Del_\real(C^T)$ which restricts to a bijection
$\Del_\rS(C,\Ome) \to \Del_\rS(C^T,\Ome)$.

\begin{Rem}
A \emph{complete real exceptional sequence} is a sequence $(\bet_1,\bet_2,\ldots,\bet_n)$ of
real roots  such that
$\bil{\bet_i,\bet_j}=0$ for $i>j$.

Recall that the \emph{braid group} $B_n$ is defined by generators $\sigma_1,\ldots,\sigma_{n-1}$ with relations $\sigma_i\sigma_j = \sigma_j\sigma_i$
for $|i-j| \ge 2$ and $\sigma_i\sigma_{i+1}\sigma_i = \sigma_{i+1}\sigma_i\sigma_{i+1}$ for $1 \le i \le n-2$.

The braid group $B_n$ acts on the
set of complete real exceptional sequences via
\[
  \sig_i(\bet_1,\cdots,\bet_n)= (\bet_1,\ldots,\bet_{i-1},s_{\bet_i}(\bet_{i+1}),\bet_i,\bet_{i+2},\ldots,\bet_n).
\]
In fact, the semidirect product $\{\pm 1\}^n\rtimes B_n$ of the sign group
with the braid group acts transitively on the set of all complete real
exceptional sequences.  The set of real Schur roots $\Del_\rS(C,\Ome)$
can be described alternatively as the set of positive roots which appear
in some complete real exceptional sequence, see \cite{HuK16}.
\end{Rem}


\section{Algebras associated with Cartan matrices} \label{sec:defH}


\subsection{Combinatorics of symmetrizable Cartan matrices}
Let $c := \lcm((c_i)_{i\in I})$ and recall the following notations
from \cite[Section~1.4]{GLS1}: For $(i,j)\in I\times I$ with $c_{ij}<0$ we define
natural numbers
\[
  g_{ij}:=\gcd(-c_{ij},-c_{ji}),\qquad f_{ij}:=-c_{ij}/g_{ij},\qquad
  k_{ij}:=\gcd(c_i,c_j),\quad l_{ij}=\lcm(c_i,c_j).
\]    
Thus we have
\[
  g_{ij}=g_{ji},\qquad k_{ij}=k_{ji},\qquad c_i=k_{ij}f_{ji},\qquad
  l_{ij}=c_ic_j/k_{ij}.
\]  
For later use we record the following elementary facts:
\begin{itemize}

\item
For $a,b \in \ZZ_{>0}$ we have
\[
\ZZ a+\ZZ b=\ZZ\gcd(a,b)
\text{\quad and \quad}
\NN a + \NN b \supset
\left(\NN + \frac{\lcm(a,b)}{\gcd(a,b)}\right)\gcd(a,b).
\]

\item
On the other hand,
\[
\lcm(\frac{c}{c_i},\frac{c}{c_j}) = \frac{c}{\gcd(c_i,c_j)}
\text{\quad and \quad}
\gcd(\frac{c}{c_i},\frac{c}{c_j}) = \frac{c}{\lcm(c_i,c_j)}.
\]

\end{itemize} 
We deduce immediately the following result:

\begin{Lem} \label{lem:elem}
With the above notation we have
\[
\ZZ\frac{1}{c_i}+\ZZ\frac{1}{c_j}=\ZZ\frac{1}{\lcm(c_i,c_j)}
\text{\quad and \quad}
\NN\frac{1}{c_i} + \NN\frac{1}{c_j} \supset
\left(\NN + 
\frac{\lcm(c_i,c_j)}{\gcd(c_i,c_j)}\right) \frac{1}{\lcm(c_i,c_j)}.
\]
\end{Lem}

\subsection{The algebras $H_\FF(C,D,\Omega)$}
We set 
\[
H_i := \FF[\eps_i]/(\eps_i^{c_i})
\] 
for $i\in I$. 
If $(i,j)\in\Ome$
we define the cyclic $H_i$-$H_j$-bimodule ${_iH'_j}$ with generator
$\alp_{ij}$ by the relation 
\[
\eps_i^{f_{ji}}\alp_{ij}=\alp_{ij}\eps_j^{f_{ij}}.
\]
Thus ${_iH'_j}$ is free of rank $f_{ij}$ as a left $H_i$-module and free of
rank $f_{ji}$ as a right $H_j$-module. 
Then we set 
\[
{_iH_j}:=({_iH'_j})^{g_{ij}}
\]
and define
\[
  S:=\prod_{i\in I} H_i\qquad\text{and}\qquad
  B:=\bigoplus_{(i,j)\in\Ome} {_iH_j}
\]
Note that $B$ is naturally an $S\text{-}S$-bimodule. So we can 
define the tensor
algebra
\[
  H := H_\FF(C,D,\Ome) := T_S(B).
\]  
We introduced and studied the algebras $H$ in \cite{GLS1}.
More precisely, we defined them there via quivers with relations and then proved that they are isomorphic to $T_S(B)$, see
\cite[Proposition~6.4]{GLS1}.

Thus, an $H$-module $M$ can be described as a tuple
\[
  (\ul{M}, (M_{ij})_{(i,j)\in\Ome}) \text{ with }
    \ul{M}=(M_i)_{i\in I}\in\rep(S),\text{ and }
    M_{ij}\in\Hom_{H_i}({_iH_j}\otimes_{H_j} M_j, M_i).
\]
We say that $M$ is \emph{locally free} if $\ul{M}$ is a projective
$S$-module, or equivalently if each $M_i$ is a free $H_i$-module. In this case we write
\[
\rkv_H(M):=(\rk_{H_i}(M_i))_{i\in I}.
\]  
Note that for example if $(i,j), (j,k)\in\Ome$, then 
${_iH_j}\otimes_{H_j}{_jH_k}$ is an $H_i\text{-}H_k$-bimodule, which is free
of rank $\abs{c_{ij}c_{jk}}$ as a left $H_i$-module, and free of rank
$\abs{c_{ji}c_{kj}}$ as a right $H_k$-module. Along this line it is easy to
see that the $H$-modules ${_HH}$ and ${_HDH}$ are  locally free. 
It follows that the category $\rep_\lf(H)$ of locally free $H$-modules is an exact category with enough projectives and injectives. 
From the definition of $H$ as the tensor algebra $T_S(B)$ and
the fact that $B$ is projective as a right and left $S$-module we
obtain a short exact sequence of $H\text{-}H$-bimodules
\[
  0\ra H\otimes_S B\otimes_S H\ra H\otimes_S H\ra H\ra 0,
\]
where each term is projective as a left $H$-module and as right $H$-module.
It yields a functorial projective resolution for all locally free $H$-modules.
On the other hand it is easy to see that only locally free modules can have
finite projective dimension. Now it is easy to derive the following result.
The next result is proved in
\cite[Proposition~3.5 and Proposition~4.1]{GLS1}.

\begin{Prop}\label{prop:GLS1}
For  $M,N\in\rep(H)$ we have:
\bit

\item[(a)]
$M$ is locally free $\iff$ $\pdim(M)\leq 1$ $\iff$ $\pdim(M) < \infty$.
   
\item[(b)]
$N$ is locally free $\iff$ $\idim(N)\leq 1$ $\iff$ $\idim(N) < \infty$.
     
\item[(c)]
If $M$ and $N$ are locally free we have
\[
\dim_\FF \Hom_H(M,N) - \dim_\FF \Ext^1_H(M,N) 
= \bil{\rkv_H(M),\rkv_H(N)}.
\]

\eit
\end{Prop}

\begin{Rem} \label{rem:runique}
We observed in the proof of \cite[Proposition~3.2]{GLS2} that for $\FF$ algebraically
  closed there exists at most one isomorphism class of rigid locally free $H$-modules with a given rank vector. 
However, it is easy to extend this result to any field $\FF$ by a weak version of the Noether-Deuring Theorem.
Namely, if $\FF\subseteq\GG$ is a field extension, we have  a natural
isomorphism
\[
H_\GG :=  H_\GG(C,D,\Ome) \cong H\otimes_\FF \GG.
\]  
Thus, $H_\GG$, viewed as an $H\text{-}H$-bimodule,
has a canonical direct summand isomorphic to $H$. It follows, that for any
$H$-module $M$  the natural monomorphism
\begin{align*}
M &\ra M\otimes_\FF\GG
\\
m &\mapsto m \otimes 1
\end{align*}
of $H$-modules splits.
The $H$-module $M\otimes_\FF\GG$ is isomorphic to a (possibly
infinite) direct sum of copies of $M$.
As a consequence, if $M,N\in\rep(H)$ are such that
\[
M\otimes_\FF\GG\cong N\otimes_\FF\GG
\] 
as $H_\GG$-modules, then $M\cong N$ as
$H$-modules. 
Here we are using the Krull-Remak-Schmidt-Azumaya Theorem, see for example Facchini's book \cite{F12}.

Now, if we take for $\GG$ the algebraic closure of $\FF$, and
$M, N\in\rep_\lf(H)$ rigid modules with $\rk_H(M)=\rk_H(N)$, we get
that $M\otimes_\FF\GG$ and $N\otimes_\FF\GG$ are rigid $H_\GG$-modules with
\[
\rk_{H_\GG}(M\otimes_\FF\GG)=\rk_{H_\GG}(N\otimes_\FF\GG). 
\]
By our observation
from \cite{GLS2} we conclude that 
\[
M\otimes_\FF\GG\cong N\otimes_\FF\GG
\] 
as $H_\GG$-modules, and thus $M\cong N$. 
\end{Rem}


\section{Species and completions}\label{sec:species}


\subsection{Standard species over the field of Laurent series $\FF\lpw{\veps}$} \label{ssec:species1}
The following is very similar to the standard construction of species over finite fields.

We fix an indeterminate which we denote by $\veps^{1/c}$. All constructions will take place in the ambient field $\FF\lpw{\veps^{1/c}}$ of formal Laurent series.
Write $\veps := (\veps^{1/c})^c$.
We consider the degree $c$ field extension $\FF\lpw{\veps}\subseteq\FF\lpw{\veps^{1/c}}$ 
of formal Laurent series. More generally, for each positive 
divisor $k$ of $c$, we set
\[
\veps^{1/k}:=(\veps^{1/c})^{c/k} \in \FF\lpw{\veps^{1/c}}.
\]
Moreover we abbreviate 
$$
\veps_i:=\veps^{1/c_i}
$$ 
for $i\in I$.

In particular, the field extension 
$\FF\lpw{\veps} \subseteq \FF\lpw{\veps_i}$ has degree $c_i$.
For $(i,j)\in\Ome$ we have the
following diagram of field extensions:
\[\xymatrix{
    &\FF\lpw{\veps_i,\veps_j}& = &\FF\lpw{\veps^{1/l_{ij}}}\\
    \FF\lpw{\veps_i}\ar@{-}[ru]^{f_{ij}}\ar@{-}[rd]_{f_{ji}}&&\FF\lpw{\veps_j}\ar@{-}[lu]_{f_{ji}}\\
    &\FF\lpw{\veps_i}\cap\FF\lpw{\veps_j}\ar@{-}[d]^{k_{ij}}\ar@{-}[ru]_{f_{ij}}&=&\FF\lpw{\veps^{1/k_{ij}}}\\
    &\FF\lpw{\veps}
  }\]  
Observe that $\FF\lpw{\veps_i,\veps_j}=\FF\lpw{\veps^{1/l_{ij}}}$ by 
the first statement of Lemma~\ref{lem:elem}. 
The claims about the degrees follow, since 
$\veps_i^{f_{ji}}=\veps_j^{f_{ij}}=\veps^{1/k_{ij}}$. Now we set
\[
  \tS:= \prod_{i\in I} \FF\lpw{\veps_i}
\text{\qquad and \qquad}
  {_i\tH_j}:=\left(\FF\lpw{\veps_i,\veps_j}\right)^{g_{ij}}.
\]
In particular,
${_i\tH_j}$ is an $\FF\lpw{\veps_i}\text{-}\FF\lpw{\veps_j}$-bimodule,
which is free of rank $|c_{ij}|$ as a left $\FF\lpw{\veps_i}$-module
and free of rank $|c_{ji}|$ as a right $\FF\lpw{\veps_j}$-module,
for $(i,j) \in \Omega$.
 Finally, we define the tensor algebra
\[
\tH :=  \tH_{\FF\lpw{\veps}}(C,D,\Ome):= T_\tS\left(\bigoplus_{(i,j)\in\Ome} {_i\tH_j}\right).
\]
Then
$\tH$ is a finite-dimensional hereditary $\FF\lpw{\veps}$-algebra.
If $C$ is connected and $D$ is the minimal symmetrizer of $C$, the center of
$\tH$ is $\FF\lpw{\veps}$, otherwise the center may be strictly larger 
than $\FF\lpw{\veps}$.
It is easy to see that the generalized Cartan lattice of the hereditary algebra $\tH$ is isomorphic
to $L(C,D,\Ome)$.
(Each finite-dimensional hereditary $\FF$-algebra 
gives rise to a generalized Cartan lattice, compare 
\cite[Section~4]{HuK16}.)

\subsection{Integral form of the species $\tH$}
We study now intermediate rings in the extension 
$\FF\pow{\veps} \subseteq \FF\pow{\veps^{1/c}}$
of formal power series rings. We think of $\FF\pow{\veps^{1/c}}$ as the
ring of integers of the field $\FF\lpw{\veps^{1/c}}$ and observe
that $\veps^{1/k}\in\FF\pow{\veps^{1/c}}$ for each positive divisor $k$ of $c$.
Note that $\FF\pow{\veps_i}$ is free of rank $c_i$ as an $\FF\pow{\veps}$-module. 
Similarly to the previous section, we obtain 
the following diagram for each $(i,j)\in\Ome$:
\[\xymatrix{
    &\FF\pow{\veps^{1/l_{ij}}}\\
    &\FF\pow{\veps_i,\veps_j}\ar@{^{(}->}[u]_{\mathrm{fin. codim}}\\
    \FF\pow{\veps_i}\ar@{-}[ru]^{f_{ij}}\ar@{-}[rd]_{f_{ji}}&&\FF\pow{\veps_j}\ar@{-}[lu]_{f_{ji}}\\
    &\FF\pow{\veps_i}\cap\FF\pow{\veps_j}\ar@{-}[d]^{k_{ij}}\ar@{-}[ru]_{f_{ij}}&=&\FF\pow{\veps^{1/k_{ij}}}\\
    &\FF\pow{\veps}
  }\]  
Here, an edge 
\[
\xymatrix{B\ar@{-}[d]^d\\A}
\] 
stands 
for an inclusion $A\subset B$ of
rings such that $B$ is free of rank $d$ as an $A$-module.
As in Section~\ref{ssec:species1}, the claims follow easily from the equation
$\veps_i^{f_{ji}}=\veps_j^{f_{ij}}=\veps^{1/k_{ij}}$. 
The subring
\[
\FF\pow{\veps_i,\veps_j} \subseteq \FF\pow{\veps^{1/l_{ij}}}
\]
has finite codimension over $\FF$ by the second statement of Lemma~\ref{lem:elem}. 
In particular we obtain, after localizing with respect to $\veps$, 
the identity
\[
\FF\pow{\veps_i,\veps_j}_\veps=\FF\pow{\veps^{1/l_{ij}}}_\veps=\FF\lpw{\veps^{1/l_{ij}}}.
\]

Now we set
\[
  \hS:= \prod_{i\in I} \FF\pow{\veps_i}
\text{\qquad and \qquad}
  {_i\hH_j}:=\left(\FF\pow{\veps_i,\veps_j}\right)^{g_{ij}}.
\]
We regard $\hS$ as an $\FF\pow{\veps}$-algebra by mapping $\veps$ to
the tuple $(\veps_i^{c_i})_{i\in I}$.
Let $\hH_i := \FF\pow{\veps_i}$.
Then ${_i\hH_j}$ is an $\hH_i$-$\hH_j$-bimodule,
which is free of rank $|c_{ij}|$ as a left $\hH_i$-module and free of
rank $|c_{ji}|$ as a right $\hH_j$-module,
for $(i,j)\in\Ome$. Finally, we define the tensor algebra
\[
\hH :=  \hH_{\FF\pow{\veps}}(C,D,\Ome):= T_\hS\left(\bigoplus_{(i,j)\in\Ome} {_i\hH_j}\right),
\]
which is an $\FF\pow{\veps}$-algebra.

Similarly to the situation for $H$, an $\hH$-module $M$ can be described as a tuple
\[
  (\ul{M}, (M_{ij})_{(i,j)\in\Ome})
\] 
with
\[
    \ul{M}=(M_i)_{i\in I}\in\rep(\hS) \text{\qquad and \qquad}
    M_{ij}\in\Hom_{\FF\pow{\eps_i}}({_i\hH_j}\otimes_{\hH_j} M_j, M_i).
\]
We say that $M\in\rep(\hH)$ is
\emph{locally free} if $M_i$ is free as an $\hH_i$-module for all
$i\in I$.

Let
\[
\hH_\veps = \hH\otimes_{\FF\pow{\veps}}\FF\lpw{\veps}
\]
be the localization of $\hH$ at $\veps$.

\begin{Prop} \label{prop:hH-basic}
The algebra $\hH = \hH_{\FF\pow{\veps}}(C,D,\Ome)$ has the following
properties:
\bit

\item[(a)]
The $\hH$-module
${_\hH}\hH$ is locally free.

\item[(b)]
  $\FF\pow{\veps}\subseteq Z(\hH)$ and $\hH$ is free of finite rank as
  an $\FF\pow{\veps}$-module.

\item[(c)]
We have
$\hH_\veps \cong \tH$.

\item[(d)]
We have
\[
\hH/(\veps^k\hH)\cong H_\FF(C,kD,\Ome)
\] 
for all $k \in \ZZ_{> 0}$.
Thus
\[
\varprojlim_k H_\FF(C,kD,\Ome) \cong \hH.
\] 

\item[(e)]
  The decomposition $1_\hS=\sum_{i\in I} e_i$ of $1_\hS$ into orthogonal
  primitive idempotents lifts via $\hS\hookrightarrow\hH$ to 
  $1_{\hH} =\sum_{i\in I} e_i$, which is also a decomposition into primitive
  orthogonal idempotents,  such that $e_i\hH e_i=\FF\pow{\veps_i}$.

\eit
\end{Prop}

\begin{proof}
  (a) follows, similarly to the situation for $H$, from the following
  observation: If $(i_1,i_2), (i_2,i_3),\ldots (i_l,i_{l+1})\in\Ome$ then
  the $\FF\pow{\veps_{i_1}}\text{-}\FF\pow{\veps_{i_{l+1}}}$-bimodule
  \[
    {_{i_1}\hH_{i_2}}\otimes_{\FF\pow{\veps_{i_2}}}{_{i_2}\hH_{i_3}}
 \otimes_{\FF\pow{\veps_{i_3}}}\cdots\otimes_{\FF\pow{\veps_{i_l}}}{_{i_l}\hH_{i_{l+1}}}
  \]  
is free of rank
$\abs{c_{i_1,i_2}\cdot c_{i_2,i_3}\cdot\,\cdots\,\cdot c_{i_l,i_{l+1}}}$,
since ${_i\hH_j}$ is free of rank $-c_{ij}$ as a left 
$\FF\pow{\veps_i}$-module when
$(i,j)\in\Ome$.

Recall that $(i,j)\in\Ome$ implies by our convention $i<j$. Thus there are
only finitely many sequences $i_1,i_2,\ldots,i_l,i_{l+1}$ in $I$ such that
$(i_1,i_2), (i_2,i_3),\ldots,(i_l,i_{l+1})\in\Ome$. 

For (b) we observe, that by construction $\veps\in Z(\hH)$. 
Now, the rest of (b) follows from (a) since $\FF\pow{\veps_i}$, viewed as
an $\FF\pow{\veps}$-module, is free of rank $c_i$ for all $i\in I$.

Next, the localization $\FF\pow{\veps}_\veps$ is obviously
$\Quot(\FF{\pow{\veps}})=\FF\lpw{\veps}$,
the field of Laurent series, and $\FF\pow{\veps_i}_\veps$ is a commutative 
$\FF\lpw{\veps}$-algebra of dimension $c_i$ without zero-divisors. Thus,
$\FF\lpw{\veps} \subseteq \FF\pow{\veps_i}_\veps$ is a field extension of degree $c_i$. 
Since $\veps_i^{c_i}=\veps$ we conclude that $\FF\pow{\veps_i}_\eps=\FF\lpw{\veps_i}$
and $\hS_\veps=\tS$. Now it is clear from the constructions that we have
a natural identification of $\tS\text{-}\tS$-bimodules
\[
  \left(\bigoplus_{(i,j)\in\Ome}{_i\hH_j}\right)_\veps =
  \bigoplus_{(i,j)\in\Ome}{_i\tH_j},
\]
which shows (c).

In order to show (d), we proceed similarly. Note first that obviously
\[
\hS/(\veps) \cong \prod_{i\in I} \FF[\eps_i]/(\eps_i^{c_i}),
\] 
where the isomorphism sends $\veps_i + (\veps)$ to $\eps_i + (\eps_i^{c_i})$.
Next,
recalling the definition of ${_iH_j'}$ from Section~\ref{sec:defH} we observe
that $\FF\pow{\veps^{1/l_{ij}}}/(\veps)$ is under $1\mapsto\alp_{ij}$ isomorphic
to ${_iH'_j}$ as an $\FF[\eps_i]/(\eps_i^{c_i})\text{-}\FF[\eps_j]/(\eps_j^{c_j})$-bimodule. This shows that $\hH/(\veps\hH)\cong H$. The proof that
\[
\hH/(\veps^k\hH)\cong H_\FF(C,kD,\Ome)
\] is similar. Finally, by part (b),
$\hH$ is complete in the $\veps$-adic topology, and thus
\[
\hH\cong \varprojlim_k \hH/(\veps^k\hH).
\] 

Part (e) is obvious.
\end{proof}

\begin{Rem} \label{Rem-cplnoethA}
By Proposition~\ref{prop:hH-basic}(b), $\hH$ is a Noetherian algebra over the complete local ring $\FF\pow{\veps}$,
in the sense of Auslander.
In particular, $\hH$ is a semiperfect ring. 
For any $M \in \rep(\hH)$ the ring $\End_\hH(M)$ is Noetherian and semiperfect, and $M$ is indecomposable if and only if $\End_\hH(M)$ is local.
It follows that $\rep(\hH)$ has the Krull-Remak-Schmidt property, see the beginning of Section~5 in \cite{Aus78}.
\end{Rem}

\begin{Prop} \label{prop:hH-pdim}
\bit

\item[(a)] 
For locally free $\hH$-modules we have a standard projective resolution.
In particular,  $M\in\rep_\lf(\hH)$ implies $\pdim(M)\leq 1$.

\item[(b)]
Each submodule of a locally free $\hH$-module is again locally free.

\item[(c)]
$\gldim(\hH)=2$ if $C\neq\diag(2,2,\ldots,2)$, otherwise $\gldim(\hH)=1$.

\eit
\end{Prop}

\begin{proof}
By the description of $\hH$ as $T_\hS(\hB)$ for
\[
\hB := \bigoplus_{(i,j)\in\Ome} \left(\FF\pow{\veps_i,\veps_j}\right)^{g_{ij}}
\] 
we obtain, as in
the case of $H$, a short exact sequence of $\hH\text{-}\hH$-bimodules
\[
  0\ra \hH\otimes_\hS\hB\otimes_\hS\hH\ra\hH\otimes_\hS\hH\ra\hH\ra 0
\]
with each term projective as a left $\hH$-module and as a right $\hH$-module.
This yields for each locally free $\hH$-module $M$ a functorial projective
resolution
\[
  0\ra \hH\otimes_\hS\hB\otimes_\hS M\ra\hH\otimes_\hS M\ra M\ra 0.
\]
In particular, we have $\pdim(M)\leq 1$.
Thus (a) is proved.

Since $\FF\pow{\veps_i}$ is a (local) principal ideal domain,
each submodule of a free $\FF\pow{\veps_i}$-module is again free.
In other words, each submodule of a projective $\hS$-module is again
projective. Thus, in particular each submodule of a locally free $\hH$-module
is locally free.
This proves (b).

Finally, $\hH$ is locally free as a left $\hH$-module by
Proposition~\ref{prop:hH-basic}(a). Thus, for each left ideal $I\leq\hH$ we
have $\pdim_\hH(I)\leq 1$ by parts (a) and (b).
By \cite[Theorem~1]{Aus55} this implies $\gldim({\hH})\leq 2$. 

Note that $C \not= \diag(2,\ldots,2)$ if and only if $\Ome \not= \emptyset$.
In this case, let $(i,j) \in \Ome$.
Then one checks easily that $\pdim(S_j) = 2$ and therefore
$\gldim(\hH) = 2$.
(Here $S_j$ is the simple $\hH$-module associated with
$j$.)
For $C = \diag(2,\ldots,2)$ we have $\hH \cong \hS$ and 
therefore
$\gldim(\hH) = 1$.
This finishes the proof of (c).
\end{proof}

Note however, that the simple $\tH$-module $S_1$, which is not locally free,
has also $\pdim(S_1)=1$. We will see later that partial tilting $\hH$-modules
are locally free.

\begin{Lem}\label{lem:free1}
For $M \in \rep(\hH)$ the following 
are equivalent:
\bit

\item[(a)]
$M$ is locally free;

\item[(b)]
$M$ is free as an $\FF\pow{\veps}$-module.

\eit
\end{Lem}

\begin{proof}
By definition
$M$ is locally free if and only if $M$ is a projective $\hS$-module.
Now $\FF\pow{\veps}$ is a subalgebra of $\hS$ and all projective $\hS$-modules are free $\FF\pow{\veps}$-modules.
This yields the result.
\end{proof}

\subsection{Orders, lattices and locally free modules}
We repeat some definitions from the theory of orders and lattices
over orders.
For more details we refer for example to
\cite[Chapter~3]{CR81}.

Let 
\[
R := \FF\pow{\veps} 
\text{\qquad and \qquad}
K := \FF\lpw{\veps}.
\]
An $R$-\emph{lattice} is a finitely generated projective $R$-module.
Thus (since we do not work with arbitrary Dedekind domains $R$ as in
\cite{CR81}),
an $R$-lattice is a free $R$-module of finite rank.

An $R$-\emph{order} is a ring $\Lambda$ such that the following hold:
\bit

\item[(i)]
The center $Z(\Lambda)$ of $\Lambda$
contains $R$;

\item[(ii)]
$\Lambda$ is an $R$-lattice.

\eit

Now let $A$ be a finite-dimensional $K$-algebra.
An $R$-\emph{order in} $A$ is a subring $\Lambda$ of $A$
such that the following hold:
\bit

\item[(i)]
$\Lambda$ is an $R$-order;

\item[(ii)]
$\Lambda$ generates $A$ as a $K$-vector space.

\eit

Now let $\Lambda$ be an $R$-order in a $K$-algebra $A$.
A $\Lambda$-\emph{lattice} is a $\Lambda$-module, which is an $R$-lattice.

\begin{Prop}\label{prop:free2}
The algebra $\hH$ is an $R$-order in the $K$-algebra $\tH$.
Furthermore, for $M \in \rep(\hH)$ the following 
are equivalent:
\bit

\item[(a)]
$M$ is an $\hH$-lattice;

\item[(b)]
$M$ is locally free.

\eit
\end{Prop}

\begin{proof}
The first part follows from the definitions.
By definition $M$ is an $\hH$-lattice if and only if $M$ is free as
an $R$-module.
Now the result follows from Lemma~\ref{lem:free1}.
\end{proof}

\begin{Lem}\label{lem:homfree}
Let $M \in \rep(\hH)$ and $N \in \rep_\lf(\hH)$.
Then
$\Hom_\hH(M,N)$ is free 
as an $R$-module.
\end{Lem}

\begin{proof}
Since $R$ is a subalgebra of the center $Z(\hH)$ of $\hH$,
we get that $M = {_\hH}M_R$ is an $\hH$-$R$-bimodule.
This turns $\Hom_\hH(M,N)$ into an $R$-module.

Next, let $\hH e_i$ be an indecomposable projective $\hH$-module.
Then there is an $R$-module isomorphism 
$\Hom_\hH(\hH e_i,N) \cong e_iN$ defined by $f \mapsto f(e_i)$.
It follows that
$\Hom_\hH(\hH e_i,N)$ is a free $R$-module of finite rank.
(Here we used that $N$ is locally free.)
By the additivity of Hom-functors, we get that $\Hom_\hH(P,N)$
is a free $R$-module of finite rank for all finitely generated projective $\hH$-modules $P$.

Let $P_0 \to M$ be a projective cover of $M$.
Appyling $\Hom_\hH(-,N)$ gives an exact sequence
$$
0 \to \Hom_\hH(M,N) \to \Hom_\hH(P_0,N).
$$
We know already that
$\Hom_\hH(P_0,N)$ is free of finite
rank as an $R$-module. 
So
$\Hom_\hH(M,N)$ is a submodule of a free $R$-module and is therefore also free as an $R$-module.
\end{proof}


\section{Reduction and localization functors}\label{sec:redloc}


\subsection{Definitions and first properties}
For the rest of the paper we set
\[
R := \FF\pow{\veps} 
\text{\qquad and \qquad}
K := \FF\lpw{\veps}.
\]
By Proposition~\ref{prop:hH-basic} we have two canonical homomorphisms 
\[
\xymatrix{\hH_R(C,D,\Ome)\ar@{^(->}[r]^{\iota}\ar@{->>}[d]_{\pi} &\tH_K(C,D,\Ome)\\  
  H_\FF(C,D,\Ome)}
\]
of $\FF$-algebras.
In particular we can view $H$ as an $H\text{-}{\hH}$-bimodule and $\tH$ as an
$\tH\text{-}\hH$-bimodule. 
It is easy to see that the full subcategory of
locally free $\hH$-modules, $\rep_\lf(\hH)$, is an exact subcategory of
$\rep(\hH)$. 
Via restrictions of the tensor product functors we obtain
the \emph{reduction functor}
\begin{align*}
\Red\df\rep_\lf(\hH)&\ra\rep_\lf(H)
\\
M & \mapsto H\otimes_\hH M 
\end{align*}
and the \emph{localization functor}
\begin{align*}
\Loc\df\rep_\lf(\hH)&\ra\rep(\tH)
\\  
M & \mapsto\tH\otimes_\hH M.
\end{align*}
We have the
identifications
\[
\Red(M) = \bar{M} := M/(\veps M)
\text{\qquad and \qquad} 
\Loc(M) = M_\veps,
\]
where $M_\veps$ is the localization with respect
to $\veps\in Z(\hH)$.

\begin{Lem}\label{lem:functors1}
Assume that $M \in \rep_\lf(\hH)$.
Then
\[
\rkv_\hH(M) = \rkv_H(\Red(M)) = \dimv(\Loc(M)).
\]
\end{Lem}

\begin{proof}
Let
$M = ((M_i)_{i \in I},(M_{ij})_{(i,j) \in \Ome}) \in \rep_\lf(\hH)$.
Then
$$
\Red(M) = 
((S \otimes_\hS M_i)_{i \in I},(S \otimes_\hS M_{ij})_{(i,j) \in \Omega})
$$
and
$$
\Loc(M) = 
((\tS \otimes_\hS M_i)_{i \in I},(\tS \otimes_\hS M_{ij})_{(i,j) \in \Omega}).
$$

We have
$$
S \otimes_\hS M_i 
\cong \FF[\eps_i]/(\eps_i^{c_i}) \otimes_{\FF\pow{\veps_i}} M_i,
$$
where $\FF[\eps_i]/(\eps_i^{c_i})$ is regarded as an $\FF\pow{\veps_i}$-module via $\veps_i \mapsto \eps_i + (\eps_i^{c_i})$.
This is a free $\FF[\eps_i]/(\eps_i^{c_i})$-module whose rank is $\rk(M_i)$.
It follows that 
\[
\rkv_\hH(M) = \rkv_H(\Red(M)).
\]
Furthermore, we have
$$
\tS \otimes_\hS M_i 
\cong \FF\lpw{\veps_i} \otimes_{\FF\pow{\veps_i}} M_i.
$$
This is an $\FF\lpw{\veps_i}$-vector space of dimension $\rk_{\hH_i}(M_i)$.
It follows that 
\[
\rkv_\hH(M) = \dimv(\Loc(M)).
\]
\end{proof}

\begin{Lem}\label{lem:exact}
The functors $\Red$ and $\Loc$ are exact.
\end{Lem}

\begin{proof}
Being tensor functors, both $\Red$ and $\Loc$ are right exact.
The exactness follows now from
Lemma~\ref{lem:functors1} due to dimension reasons.
\end{proof}

\begin{Lem}\label{lem:Proj-bij}
$\Red$ and $\Loc$ induce bijections
\[
\xymatrix{
\{ \text{projective $\hH$-modules} \}/\!\!\cong \;\;
\ar[r]\ar[d] & \;\;
\{ \text{projective $\tH$-modules} \}/\!\!\cong 
\\
\{ \text{projective $H$-modules} \}/\!\!\cong .
}
\]
\end{Lem}

\begin{proof}
We have
\[
\Red(\hH e_i) = (\hH e_i)/(\veps \hH e_i) \cong H e_i
\text{\qquad and \qquad}
\Loc(\hH e_i) = (\hH e_i)_\veps \cong \tH e_i.
\]
Now the result follows by additivity.
\end{proof}

\begin{Lem}\label{lem:dense}
The functors $\Red$ and $\Loc$ are dense.
\end{Lem}

\begin{proof}
Let $M = ((M_i)_{i \in I},(M_{ij})_{(i,j) \in \Omega}) \in \rep_\lf(H)$.
We can assume that for each $i$ we have 
$M_i = H_i^{m_i}$ for some $m_i \ge 0$.
Then the $H_i$-linear map $M_{ij}\df {_i}H_j \otimes_{H_j} H_j^{m_j} \to H_i^{m_i}$
can be lifted to an $\hH_i$-linear map
$\widehat{M}_{ij}\df {_i}\hH_j \otimes_{\hH_j} \hH_j^{m_j} \to \hH_i^{m_i}$.
This yields a representation $\widehat{M} \in \rep_\lf(\hH)$ with
$\Red(\widehat{M}) \cong M$.
Thus $\Red$ is dense.

On the other hand, the denseness of $\Loc$ is a special case of
\cite[Proposition~23.16]{CR81}.
\end{proof}

\subsection{Example}\label{subsec:example} 
Consider
\[
C=\begin{pmatrix} 2 & -1\\ -1&2\end{pmatrix},
\qquad
D=\begin{pmatrix} 1 & 0\\ 0& 1\end{pmatrix}
\]
and $\Ome=\{(1,2)\}$.
Thus, $\hH_R(C,D,\Ome)$ is isomorphic to the completed path algebra over the field $\FF$ of the quiver with relations 
  \[
    \xymatrix{\ar@(ul,dl)[]_{\veps_1} 1&\ar[l]_\alp 2\ar@(ur,dr)[]^{\veps_2}}
    \quad\text{and}\quad \veps_1\alp-\alp\veps_2,
  \]  
  whilst $\tH_K(C,D,\Ome)$ is the path algebra over the field $K$
  of the quiver $\xymatrix{1&\ar[l]_\alp 2}$, and $H_\FF(C,D,\Ome)$ is the
  path algebra over $\FF$ for the same quiver.
  Now, a locally free $\hH$-module $M$ of rank $(1,1)$ is determined precisely
   by an element $m\in\FF\pow{\veps_1}$, which defines the ``structure map'' in
\[
  \Hom_{\FF\pow{\veps_1}}({_1\hH_2}\otimes_{\FF\pow{\veps_2}}\FF\pow{\veps_2},\FF\pow{\veps_1})
\]
where
\[ 
{_1\hH_2}
  \otimes_{\FF\pow{\veps_2}}\FF\pow{\veps_2} \cong \FF\pow{\veps_1}
\]
as $\FF\pow{\veps_1}$-modules.
In this sense the $\veps_1^k$ with $k \ge 0$  represent the isoclasses of
indecomposable $\hH$-modules with rank $(1,1)$.
Let $M(k)$ be the indecomposable $\hH$-module 
represented by $\veps_1^k$.
Note that $M(k)$ is projective if and only if $k=0$. 
Now
$\Loc(M(k))$ is isomorphic to the projective $\tH$-module $\tH e_2$
for all $k \ge 0$.
On the other hand, we have 
$$
\Red(M(k)) \cong
\begin{cases}
He_2 & \text{if $k = 0$},
\\
E_1 \oplus E_2 & \text{otherwise}.
\end{cases}
$$
(Note that in this example we have $E_i = S_i$ for $i=1,2$.)

The above example shows the following:
\bit

\item
Neither $\Red$ nor $\Loc$ respects isomorphism classes of modules.

\item
$\Red$ does not preserve indecomposability.
(To see this, take $\Red(M(k))$ for $k \ge 1$.)

\item
$\Loc$ might map a non-rigid module to a rigid module.
(Take $\Loc(M(k))$ for $k \ge 1$.)

\item
Neither $\Red$ nor $\Loc$ is full.
(Take $\End_\hH(M(k)) \to \End_H(\Red(M(k)))$ for $k \ge 1$ and 
take any $\End_\hH(M(k)) \to \End_\tH(\Loc(M(k)))$.)

\item
$\Red$ is not faithful.
(Take any $\End_\hH(M(k)) \to \End_H(\Red(M(k)))$.)

\item
In view of Lemma~\ref{lem:homfree} let us remark that there
are $M,N \in \rep_\lf(\hH)$ such that
$\Ext_\hH^1(M,N)$ is not free
as an $R$-module.
(For example, take $M = N = M(1)$. 
In this case, one easily checks that
$\dim_\FF \Ext_\hH^1(M,M) = 1$. So $\Ext_\hH^1(M,M)$ cannot
be a free $R$-module.)

\eit
%

\subsection{Properties of the reduction functor}

\begin{Lem}\label{lem:functorRed1}
Let $M,N \in \rep_\lf(\hH)$, and let 
$$
0 \to P_1 \to P_0 \to M \to 0
$$
be a projective resolution of $M$.
This yields a projective resolution
$$
0 \to \bar{P_1} \to \bar{P_0} 
\to \bar{M} \to 0
$$
of $\bar{M}$ and a commutative diagram 
$$
\xymatrix@-0.6pc{
0 \ar[r] & \Hom_\hH(M,N) \ar[r]\ar[d]^{\Red_{M,N}} & 
\Hom_\hH(P_0,N) \ar[r]\ar[d]^{\Red_{P_0,N}} & 
\Hom_\hH(P_1,N) \ar[r]\ar[d]^{\Red_{P_1,N}} & 
\Ext_\hH^1(M,N) \ar[r]\ar[d]^{\Red_{M,N}^1} & 0
\\
0 \ar[r] & \Hom_H(\bar{M},\bar{N}) \ar[r] & 
\Hom_H(\bar{P_0},\bar{N}) \ar[r] & 
\Hom_H(\bar{P_1},\bar{N}) \ar[r] & 
\Ext_H^1(\bar{M},\bar{N}) \ar[r] & 0
}
$$
of $R$-linear natural maps
with exact rows.
Furthermore, the following hold:
\bit

\item[(a)]
$\Red_{P_0,N}$, $\Red_{P_1,N}$ and
$\Red_{M,N}^1$ are surjective.

\item[(b)]
If $\Ext_H^1(\bar{M},\bar{N}) =0$, then
$\Ext_\hH^1(M,N) = 0$ and $\Red_{M,N}$ is surjective.

\item[(c)]
$\Ext^1_\hH(M,N)=0$ if and only if 
$\Ext^1_H(\bar{M},\bar{N})=0$.

\eit
\end{Lem}

\begin{proof}
The proof is almost identical to the proof of \cite[Proposition~2.2(c),(d)]{GLS2}.
\end{proof}

\begin{Cor}\label{cor:Redrigid}
For $M \in \rep_\lf(\hH)$ the following are equivalent:
\bit

\item[(a)]
$M$ is rigid;

\item[(b)]
$\Red(M)$ is rigid.

\eit
\end{Cor}

\begin{Cor}\label{cor:hH-unique}
Let $M \in\rep_\lf(\hH)$ be rigid.
Then 
\[
\End_H(\Red(M)) \cong \End_\hH(M)/(\veps \End_\hH(M)).
\]
\end{Cor}

\begin{proof}  
By Lemma~\ref{lem:functorRed1}(b)
there is a surjective $R$-algebra homomorphism
\[
\Red_{M,M}\df \End_\hH(M) \to \End_H(\Red(M)).
\]
By construction we have 
\[
\Ker(\Red_{M,M}) = \veps\End_\hH(M).
\]
The result follows.
\end{proof}

\begin{Lem} \label{lem:re-basic}
For $M,N\in\rep_\lf(\hH)$ the natural map
\[
\Red_{M,N}^1\df \FF \otimes_R \Ext^1_\hH(M,N) \to \Ext^1_H(\Red(M),\Red(N)) 
\]
is an isomorphism of $\FF$-vector spaces.
\end{Lem}

\begin{proof}
As in Lemma~\ref{lem:functorRed1} there is 
an exact sequence 
\begin{equation*}
0 \to \Hom_\hH(M,N) \to \Hom_\hH(P_0,N) \to 
\Hom_\hH(P_1,N) \to \Ext_\hH^1(M,N) \to 0.
\end{equation*}
Since $\FF$ is not flat as an $R$-module, we can only use 
the right exactness of tensor product functors to obtain
an exact sequence
\begin{equation}\label{eq:1}
\FF \otimes_R \Hom_\hH(P_0,N) \to 
\FF \otimes_R \Hom_\hH(P_1,N) \to 
\FF \otimes_R \Ext_\hH^1(M,N) \to 0.
\end{equation}
It is straightforward to check that
there is a natural isomorphism
\[
\FF \otimes_R \Hom_\hH(P_i,N) \cong \Hom_H(\Red(P_i),\Red(N))
\]
of $\FF$-vector spaces
for $i = 0,1$.
(Here one uses that for an indecomposable $\hH$-module
$\hH e_i$ we have $\Red(\hH e_i) \cong H e_i$.)
Using the exact sequence (\ref{eq:1}) 
and the commutative diagram in Lemma~\ref{lem:functorRed1}
yields a 
natural isomorphism
\[
\FF \otimes_R \Ext_\hH^1(M,N) \cong \Ext_H^1(\Red(M),\Red(N)).
\]
\end{proof}

\begin{Lem} \label{lem:rig-indec}
$\Red$ maps indecomposable rigid modules to indecomposable 
rigid modules.
\end{Lem}

\begin{proof}
Let $M \in \rep_\lf(\hH)$ be indecomposable rigid.
Then $\End_\hH(M)$ is a local ring, since $\rep(\hH)$ is
a Krull-Remak-Schmidt category.
Now Lemma~\ref{lem:functorRed1}(b) implies
that $\End_H(\Red(M))$ is isomorphic to a
factor ring of $\End_\hH(M)$.
Thus $\End_H(\Red(M))$ is also local, which implies that $\Red(M)$
is indecomposable.
We know already from Corollary~\ref{cor:Redrigid} that $\Red$ preserves rigidity.
\end{proof}

\begin{Lem} \label{lem:hH-unique}
Let $M, N\in\rep_\lf(\hH)$ with $\Ext_\hH^1(M,N) = 0$.
Then the natural map
\[
\Red_{M,N}\df \FF \otimes_R
\Hom_\hH(M,N) \to \Hom_H(\Red(M),\Red(N))
\]
is an isomorphism of $\FF$-vector spaces.
\end{Lem}

\begin{proof}  
We know from Lemma~\ref{lem:homfree} that $\Hom_\hH(M,N)$ is a free $R$-module.
It is straightforward to check that its rank is 
\[
h := \bil{\rkv_\hH(M),\rkv_\hH(N)}.
\] 
By Lemma~\ref{lem:functorRed1}(c) we have moreover 
$\Ext^1_H(\Red(M),\Red(N))=0$. 
Thus we have 
\[
h = \dim_\FF \Hom(\Red(M),\Red(N)) 
\]
which implies the claim for dimension reasons.
\end{proof}

\begin{Prop}\label{prop:hH-uniqueb}
If $M,N\in\rep_\lf(\hH)$ are rigid with $\rkv_\hH(M)=\rkv_\hH(N)$, then $M\cong N$.
\end{Prop}

\begin{proof}
The $H$-modules
$\Red(M)$ and $\Red(N)$ are rigid and locally free with 
\[
\rkv_H(\Red(M)) = \rkv_H(\Red(N))
\] 
by Lemma~\ref{lem:functorRed1}(c). 
Thus, $\Red(M) \cong \Red(N)$ 
by Remark~\ref{rem:runique},
and again by Lemma~\ref{lem:functorRed1}(c)
we have $\Ext^1_\hH(M,N)=0$.
Thus by Lemma~\ref{lem:hH-unique}, the map
\[
\Red_{M,N}\df\FF\otimes_R \Hom_\hH(M,N)\ra\Hom_H(\Red(M), \Red(N))
\]
is surjective. 
So, if $\vph\in\Hom_H(\Red(M),\Red(N))$ is an isomorphism, we can find
$\widehat{\vph}\in\Hom_\hH(M,N)$
with 
\[
\Red_{M,N}(1 \otimes \widehat{\vph})=\vph.
\] 
Then $\widehat{\vph}$ is the requested isomorphism.
In fact, we have $\vph=(\vph_i)_{i\in I}$. 
After choosing bases, we may think of $\vph_i$ as an element in $\Mat_{m_i\times m_i}(\FF[\eps_i]/(\eps_i^{c_i}))$. 
Similarly, we have then $\widehat{\vph}=(\widehat{\vph}_i)_{i\in I}$ with
$\widehat{\vph}_i\in\Mat_{m_i\times m_i}(\FF\pow{\veps_i})$ for all 
$i\in I$.
Since $\FF[\eps_i]/(\eps_i^{c_i})$ and $\FF\pow{\veps_i}$ are both local rings with residue field $\FF$, $\vph_i$ resp. $\widehat{\vph}_i$ are invertible if and only if their respective reductions modulo $\eps_i$ and $\veps_i$ in
  $\Mat_{m_i\times m_i}(\FF)$ are invertible. 
Moreover, since $\widehat{\vph}$ is the lift of $\vph$, those reductions must coincide for all $i\in I$.
\end{proof}

\begin{Prop} \label{prop:Red-bij}
$\Red$ induces a bijection 
\[
\xymatrix{
\{ \text{rigid locally free $\hH$-modules} \}/\!\!\cong 
\ar[d]^{\Red} 
\\
\{ \text{rigid locally free $H$-modules} \}/\!\!\cong .
}
\]
This bijection and its inverse preserve indecomposability.
\end{Prop}

\begin{proof}
Thanks to Corollary~\ref{cor:Redrigid}, $\Red$  induces a well defined map between the respective isoclasses of rigid locally free modules.
Now
Lemma~\ref{lem:functors1} combined with
Proposition~\ref{prop:hH-uniqueb} implies that this map is injective.
Combining Lemma~\ref{lem:dense} and Corollary~\ref{cor:Redrigid}
we get that the map is surjective. 
This yields the desired bijection.
 
By Lemma~\ref{lem:rig-indec}, the functor $\Red$ maps 
indecomposable rigids to indecomposable rigids.
Now the Krull-Remak-Schmidt property for $\rep(\hH)$ 
and $\rep(H)$ together with
the additivity of $\Red$ implies that the inverse of the bijection also
preserves indecomposability.
\end{proof}

\subsection{Properties of the localization functor}

\begin{Lem} \label{lem:functors3}
For $M, N\in\rep_\lf(\hH)$ the natural maps
\[
\Loc_{M,N}\df K \otimes_R \Hom_\hH(M,N) \to \Hom_\tH(\Loc(M),\Loc(N))
\]
and
\[
\Loc_{M,N}^1\df K \otimes_R \Ext_\hH^1(M,N) \to 
\Ext_\tH^1(\Loc(M),\Loc(N))
\]
are isomorphisms of $K$-vector spaces.
\end{Lem}

\begin{proof}
The field $K$ is flat as an $R$-module.
Now the result is just a special case of the Change of Rings Theorem, see for example \cite[Theorem~8.16]{CR81}.
\end{proof}

\begin{Cor}\label{cor:Locrigid}
$\Loc$ preserves rigidity.
\end{Cor}

\begin{Cor}\label{cor:faithful}
$\Loc$ is faithful.
\end{Cor}

\begin{proof}
For $M,N \in \rep_\lf(\hH)$ we get an injective map
\begin{align*}
\Hom_\hH(M,N) &\to K \otimes_R \Hom_\hH(M,N)
\\
f &\mapsto 1 \otimes f,
\end{align*}
since $R$ can be seen as a subring of $K$ and since
$\Hom_\hH(M,N)$ is a free $R$-module by
Lemma~\ref{lem:homfree}.
Now the result follows from the first part of Lemma~\ref{lem:functors3}.
\end{proof}

We need the following straightforward lemma.

\begin{Lem}\label{lem:Aeps}
Let $A = R^m$ and $A_\eps = K^m$.
Then $V \mapsto V_\eps$ and $U \cap A \mapsfrom U$
define mutually inverse bijections 
\[
\{ \text{direct summands of the $R$-module $A$} \} \leftrightarrow{}
\{ \text{subspaces of the $K$-vector space $A_\eps$} \}.
\]
\end{Lem}

\begin{Lem}\label{lem:Locindec}
$\Loc$ preserves indecomposability.
\end{Lem}

\begin{proof}
Let $M \in \rep_\lf(\hH)$ be indecomposable.
The finite-dimensional $K$-algebra 
$\End_\tH(\Loc(M))$ can be identified with the localization
$$
E_\veps := K \otimes_R E 
$$ 
of $E := \End_\hH(M)$.
By Lemma~\ref{lem:Aeps}
a non-trivial decomposition 
\[
E_\veps = I_1\oplus I_2
\] 
of $E_\eps$-modules would yield a similar decomposition of $E$, which is impossible. 
It follows that $E_\veps$ is also local.
This implies that $\Loc(M)$ is indecomposable.
\end{proof}

\begin{Prop}\label{prop:Loc-inj}
$\Loc$ induces an injection 
\[
\xymatrix{
\{ \text{rigid locally free $\hH$-modules} \}/\!\!\cong \;\;
\ar[r]^<<<<{\Loc} &\;\;
\{ \text{rigid $\tH$-modules} \}/\!\!\cong .
}
\]
This injection preserves indecomposability.
\end{Prop}

\begin{proof}
The functor $\Loc$ preserves rigidity by
Corollary~\ref{cor:Locrigid}, thus $\Loc$ induces a well defined map
between the mentioned isoclasses.
Now
Lemma~\ref{lem:functors1} combined with
Proposition~\ref{prop:hH-uniqueb} yields the desired injection.

By Lemma~\ref{lem:Locindec}, the functor $\Loc$ preserves indecomposability.
\end{proof}

At this stage we can only show that the map
from Proposition~\ref{prop:Loc-inj} is injective.
The surjectivity of this map will be proved in 
Section~\ref{sec:tilting}.


\section{Partial tilting modules and exchange graphs}\label{sec:tilting}


\subsection{Tilting and $\tau$-tilting pairs}
Let $A$ be a Noetherian algebra over a field.
We recall some definitions and results from tilting and $\tau$-tilting
theory.

A finitely generated $A$-module $T$ is a (classical) \emph{tilting module}
if the following hold:
\bit

\item[(i)]
$\pdim(T) \le 1$;

\item[(ii)]
$T$ is rigid, i.e. $\Ext_A^1(T,T) = 0$;

\item[(iii)]
There exists a short exact sequence
\[
0 \to A \to T' \to T'' \to 0
\]
with $T',T'' \in \add(T)$.

\eit
Moreover, $T$ is a (classical) \emph{partial tilting module}
if the conditions (i) and (ii) hold.

Assume from now on that $A$ is finite-dimensional.

A partial tilting module $T \in \rep(A)$ is a
tilting module if and only if $|T| = |A|$.

A pair $(T,P)$ is a \emph{support tilting pair} for $A$,
if the following hold:
\bit

\item[(i)]
$T \in \rep(A)$ is a basic partial tilting module;

\item[(ii)]
$P \in \rep(A)$ is a basic projective module with $\Hom_A(P,T) = 0$;

\item[(iii)]
$|T| + |P| = |A|$.

\eit

An $A$-module $M \in \rep(A)$ is \emph{$\tau$-rigid}, if
$\Hom_A(M,\tau(M)) = 0$.
Here $\tau$ denotes the Auslander-Reiten translation for $A$.
A pair $(T,P)$ is a \emph{support $\tau$-tilting pair} for $A$, if
the following hold:
\bit

\item[(i)]
$T \in \rep(A)$ is a basic $\tau$-rigid module;

\item[(ii)]
$P \in \rep(A)$ is a basic projective module with $\Hom_A(P,T) = 0$;

\item[(iii)]
$|T| + |P| = |A|$.

\eit

\begin{Rem} \label{rem:tau-rig}
Let $M \in \rep(A)$ be $\tau$-rigid.
By the Auslander-Reiten formula we have $\Ext^1_A(M,V) = 0$ for all
$V \in \fac(M)$. 
Thus $\tau$-rigid modules are in particular rigid.
Conversely, a 
partial tilting module $M\in\rep(A)$ is $\tau$-rigid, since
by definition $\pdim(M)\leq 1$ and $\Ext^1_A(M,M)=0$.
\end{Rem}

\subsection{Tilting modules for $H$}
The following lemma is due to Demonet \cite{D18}.
We thank him for his permission to include his lemma and his proof
into this article.

\begin{Lem}[Demonet] \label{lem:demonet}
Each $\tau$-rigid $H$-module is locally free, and thus a partial tilting module.
\end{Lem}  

\begin{proof}
  Let $X\in\rep(H)$ be \emph{not} locally free.
  By Remark~\ref{rem:tau-rig} it is sufficient to find a quotient $Z$
  of $X$ and a non-split exact sequence $0\ra Z\ra Y\ra X\ra 0$. To this end
  let
  \[
    m:=\max\{i\in I\mid X_i\text{ not free as an } H_i\text{-module}\},
  \]
  and consider $S_m$ as the simple $H_m$-module. 
Choose $p\in\Hom_{H_m}(X_m,S_m)\setminus\{0\}$ and define 
$Z = ((Z_i)_{i \in I},(Z_{ij})_{(i,j) \in \Omega})$ by
\[
    Z_i :=\begin{cases} X_i &\text{if } i>m,\\
                     S_m &\text{if } i=m,\\
                     0   &\text{otherwise;}
       \end{cases}\qquad
    Z_{ij} :=\begin{cases} X_{ij}      &\text{if } i>m,\\
                       p\circ X_{mj} &\text{if } i=m,\\
                       0           &\text{otherwise.}
        \end{cases}
\]
  Thus, $Z$ is a quotient of $X$ since  we assume that $(i,j)\in\Ome$
  implies $i<j$.

  Since $X_m$ is not a free $H_m$-module, we can find a non-split short
  exact sequence 
  \[
    0 \ra S_m \xrightarrow{\alp} X'_m\xrightarrow{\bet} X_m\ra 0
  \]
of $H_m$-modules.
For $(m,i) \in \Omega$, by our assumption ${_mH_i}\otimes_{H_i} X_i$ is a free $H_m$-module
  and we can find a lift
  \[\xymatrix{
          &        & &\ar@{-->}[ld]_{X'_{mi}}{_mH_i}\otimes X_i\ar[d]^{X_{mi}}\\
   0\ar[r]&   S_m\ar[r]^{\alp}& X'_m\ar[r]^{\bet} &X_m\ar[r] & 0.
 }\]
Now we can define the module $Y\in\rep(H)$ by
\[
Y_i :=\begin{cases} X'_m &\text{if } i=m,\\
                      X_i\oplus Z_i &\text{otherwise;}
        \end{cases}\qquad
Y_{ij} :=\begin{cases} \bbsm X_{ij} & 0\\ 0& Z_{ij}\besm &\text{if }
    m\not\in\{i,j\},\\
                        \bbsm X'_{mj}, &\alp \circ Z_{mj}\besm &\text{if } i=m,\\
        \bbsm X_{im}\circ (\bbo_{_iH_m}\otimes\bet)\\ 0 \besm&\text{if } j=m.
      \end{cases}
\]
We get for all $i\in I$ a short exact sequence of $H_i$-modules
\[
  0 \ra Z_i\xrightarrow{f_i} Y_i\xrightarrow{g_i} X_i \ra 0
\]
where
\[
  f_i= \begin{cases} \bbsm 0\\ \bbo_{Z_i} \besm &\text{if } i>m,\\
                     \alp                       &\text{if } i=m,\\
                     0                        &\text{otherwise;}
       \end{cases}\qquad              
  g_i=\begin{cases} \bbsm \bbo_{X_i},& 0\besm &\text{if } i>m,\\
                      \bet                    &\text{if } i=m,\\
                      \bbo_{X_i}             &\text{otherwise.}
         \end{cases}                 
\]
Note, that for $i=m$ this sequence of $H_m$-modules does not split by
construction. It is a straightforward exercise  to show that 
\[
f:=(f_i)_{i\in I}\in\Hom_H(Z,Y)
\text{\qquad and \qquad}
g:=(g_i)_{i\in I}\in\Hom_H(Y,X).
\]
It follows that
\[
  0\ra Z\xrightarrow{f} Y\xrightarrow{g} X \ra 0
\]
is the requested non-split short exact sequence of $H$-modules.
\end{proof}

\begin{Cor}\label{cor:Demonet}
For $T \in \rep(H)$ the following are equivalent:
\bit

\item[(a)]
$T$ is a partial tilting module;

\item[(b)]
$T$ is rigid and locally free;

\item[(c)]
$T$ is $\tau$-rigid.

\eit
\end{Cor}

\begin{proof}
Combine Remark~\ref{rem:tau-rig}, Lemma~\ref{lem:demonet} and Proposition~\ref{prop:GLS1}.
\end{proof}

It remains an open problem if all rigid $H$-modules are locally
free.

We denote by $\cT(H)$ the exchange graph of support tilting 
pairs
for $H$. Its vertices are the isoclasses of support tilting pairs
$(T,P)$. 
Two different vertices $(T,P)$ and $(T',P')$ are
joined by an edge if and only if 
the basic $H$-modules $T \oplus P$ and $T' \oplus P'$ 
have
(up to isomorphism) exactly $n-1$ common indecomposable
direct summands.
(We refer to \cite[Section~2.3]{AIR14} for a more detailed 
explanation.)

\begin{Prop}
The exchange graph $\cT(H)$ of $H$ is $n$-regular.
\end{Prop}

\begin{proof}
By \cite[Theorem~2.18]{AIR14}
the exchange graph of support $\tau$-tilting pairs for $H$ is $n$-regular.
(This holds in fact for any finite-dimensional algebra with
$n$ isoclasses of simple modules.)
Now the statement follows
from Corollary~\ref{cor:Demonet}.
\end{proof}

\subsection{Partial tilting modules for $\hH$}\label{sec:hH-tilt}

By Remark~\ref{Rem-cplnoethA}, $\rep(\hH)$ 
has the Krull-Remak-Schmidt property.
By the same remark we can
apply \cite[Corollary~3.7.11]{CF04} to conclude that a partial tilting module
$T\in\rep(H)$ is a tilting module if and only if $\abs{T}=\abs{I}$.
(Note that in \cite{CF04} tilting modules are by definition finitely presented. Since $\hH$ is Noetherian, all finitely generated
$\hH$-modules are also finitely presented.) 

\begin{Lem} \label{lem:hH-ptilt}
For $T \in \rep(\hH)$ the following are equivalent:
\bit

\item[(a)]
$T$ is a partial tilting module;

\item[(b)]
$T$ is rigid and locally free. 

\eit
\end{Lem}  

\begin{proof}
Recall that for a partial tilting module $T$ we have 
$\Ext^1_\hH(T,X) = 0$ for
all $X\in\fac(T)$. Thus we can proceed as in the proof of
Proposition~\ref{lem:demonet}: If $T$ were not locally free, we
find a factor module
$X$ of $T$ together with a non-split exact sequence $0\ra X\ra Y\ra T\ra 0$,
a contradiction.

Conversely, if $T$ is locally free, we have $\pdim(T)\leq 1$ by
Proposition~\ref{prop:hH-pdim}(a).
\end{proof}

\subsection{Isomorphisms of exchange graphs}
\label{subsec:exgraph}
In view of Section~\ref{sec:hH-tilt}
we define a \emph{support tilting pair}
for $\hH$ to be a pair $(T,P)$, where $T\in\rep_\lf(\hH)$ is basic 
rigid, and $P\in\rep_\lf(\hH)$ is basic projective
such that $\Hom_\hH(P,T)=0$ and $\abs{T}+\abs{P}=\abs{I}$.

\begin{Prop} \label{prop:RcT-iso}
$\Red$ induces an isomorphism 
\[
\Red\df \cT(\hH) \xrightarrow{\sim} \cT(H)
\]
between the exchange graph of support tilting pairs for $\hH$ and
the exchange graph of support tilting pairs for $H$. 
In particular, $\cT(\hH)$ is an $\abs{I}$-regular graph.   
\end{Prop}

\begin{proof}
By Lemma~\ref{lem:hH-ptilt}
the $T\in\rep_\lf(\hH)$ which are rigid are precisely the partial
tilting modules for $\hH$. For such $T$ also $\Red(T)\in\rep_\lf(H)$ is a partial
tilting module for $H$, and $\abs{T}=\abs{\Red(T)}$. If $P\in\rep_\lf(\hH)$ is
projective we have in particular $\Ext^1_\hH(P,T)=0$ and thus
$\Hom_\hH(P,T)=0$ if and only if $\Hom_H(\Red(P), \Red(T))=0$ by
Lemma~\ref{lem:hH-unique}. We conclude, that a pair $(T,P)$ in
$\rep_\lf(\hH)$ is a support tilting pair if and only if $(\Red(T),\Red(P))$
is a support tilting pair for $H$. 
The result follows now from
Propositions~\ref{prop:hH-uniqueb}
and \ref{prop:Red-bij}.
\end{proof}

\begin{Prop}\label{prop:EcT-iso}
$\Loc$ induces an isomorphism 
\[
\Loc\df \cT(\hH) \xrightarrow{\sim} \cT(\tH)
\]
between the exchange graph of support tilting pairs for $\hH$ and
the exchange graph of support tilting pairs for $\tH$. 
In particular, $\cT(\hH)$ is a connected, $\abs{I}$-regular graph.   
\end{Prop}  

\begin{proof}
Arguing as in the preparation for Proposition~\ref{prop:RcT-iso} we conclude that $\Loc$ induces an injective map of graphs from $\cT(\hH)$ to $\cT(\tH)$.

Now, 
as a consequence of \cite[Theorem~19]{Hu11},
the graph $\cT(\tH)$ is connected and $\abs{I}$-regular. 
Since $\cT(\hH)$ is $\abs{I}$-regular by 
Proposition~\ref{prop:RcT-iso}, this injection must be also surjective.
\end{proof}

\begin{Cor}\label{cor:Tilt-bij}
The graph isomorphisms
\[
\xymatrix{
\cT(\hH)
\ar[r]^{\Loc} \ar[d]^{\Red} & \;\;
\cT(\tH)
\\
\cT(H)
}
\]
induce bijections
\[
\xymatrix{
\{ \text{tilting $\hH$-modules} \}/\!\!\cong  \;\;
\ar[d]^{\Red} \ar[r]^{\Loc} 
&\;\;
\{ \text{tilting $\tH$-modules} \}/\!\!\cong
\\
\{ \text{tilting $H$-modules in $\rep_\lf(H)$} \}/\!\!\cong .
}
\]
\end{Cor}

\begin{proof}
The functors $\Red$ and $\Loc$ preserve isomorphism classes 
of projectives, see Lemma~\ref{lem:Proj-bij}.
In particular, $(T,0)$ is a support tilting pair for $\hH$ if and only
if $(\Red(T),0)$ and $(\Loc(T),0)$ are support tilting pairs for
$H$ and $\tH$, respectively.
\end{proof}

\begin{Cor}\label{cor:Loc-bij}
$\Loc$ induces a bijection 
\[
\xymatrix{
\{ \text{rigid locally free $\hH$-modules} \}/\!\!\cong \;\;
\ar[r]^<<<<{\Loc} &\;\;
\{ \text{rigid $\tH$-modules} \}/\!\!\cong .
}
\]
This bijection and its inverse preserve indecomposability.
\end{Cor}

\begin{proof}
We know already from Proposition~\ref{prop:Loc-inj} 
that the map in the statement is injective.
Let $M \in \rep(\tH)$ be rigid.
Without loss of generality we can assume that $M$ is basic.
Then there exists some rigid $C \in \rep(\tH)$ such that
$T := M \oplus C$ is a basic tilting module.
By Proposition~\ref{prop:EcT-iso} there is some some basic tilting
module $\widehat{T} \in \rep(\hH)$ with
$\Loc(\widehat{T}) \cong T$.
(Here we used Corollary~\ref{cor:Tilt-bij}.)
Now we decompose $\widehat{T} = \widehat{T}_1 \oplus \cdots \oplus
\widehat{T}_n$ and $T = T_1 \oplus \cdots \oplus T_n$ into
indecomposables.
Now we use Lemma~\ref{lem:Locindec} and 
Proposition~\ref{prop:Loc-inj}
to see that there is a permutation
$\sigma$ such that $\Loc(\widehat{T}_i) \cong T_{\sigma(i)}$
for $1 \le i \le n$.
In particular, there is some rigid locally free $\widehat{M} 
\in \add(\widehat{T})$ with $\Loc(\widehat{M}) \cong M$.
This yields the desired bijection.

By Lemma~\ref{lem:Locindec}, the functor $\Loc$ preserves
indecomposability.
Now the Krull-Remak-Schmidt property for $\rep(\hH)$ and
$\rep(\tH)$ together with
the additivity of $\Loc$ implies that the inverse of the bijection also
preserves indecomposability.
\end{proof}

Combining 
Proposition~\ref{prop:Red-bij} and 
Corollary~\ref{cor:Loc-bij}
finishes the proof of Theorem~\ref{thm:mainresult1}.


\section{{Proof of Theorem~\ref{thm:mainresult2}}}\label{sec:proof}


\subsection{{Proof of Theorem~\ref{thm:mainresult2}(a)}}
\label{subsec:proof2}
Let $\tH := \tH_{K}(C,D,\Ome)$.
Recall, that by the results of Crawley-Boevey \cite{CB93} and Ringel \cite{Rin94} we have a bijection
\begin{align*}
\{ \text{indecomposable rigid $\tH$-modules} \}/\!\! \cong &\to
\Del_\rS(C,\Ome)
\\ 
M &\mapsto \dimv(M) = (\dim_{\FF\lpw{\veps_i}} M_i)_{i\in I}.
\end{align*}
In particular, if $M\in\rep(\tH)$ is indecomposable rigid, then
\[
\bil{\dimv(M),\dimv(M)}=c_i
\] 
for some $i\in I$. 
Moreover, in this case $\End_\tH(M) \cong \FF\lpw{\veps_i}$.

Now, if $M\in\rep_\lf(H)$ is indecomposable rigid, by
Proposition~\ref{prop:Red-bij} there exists an, up to isomorphism unique,
indecomposable rigid $\widehat{M}\in\rep_\lf(\hH)$ with 
$\Red(\widehat{M}) \cong M$, and we know that
$\rkv_\hH(\widehat{M}) = \rkv_H(M)$. 
By Corollary~\ref{cor:Locrigid} and Lemma~\ref{lem:Locindec}
also
$\Loc(\widehat{M})\in\rep(\tH)$ is indecomposable rigid, and we
know that
$\rkv_\hH(\widehat{M}) = \dimv(\Loc(\widehat{M}))$.
Thus $\rkv_H(M) \in \Del_\rS(C,\Ome)$.

Conversely, if $\alp \in \Del_\rS(C,\Ome)$ there exists a, up to isomorphism unique, 
indecomposable rigid representation $N\in\rep(\tH)$ with $\dimv_\tH(N)=\alp$. 
By Lemma~\ref{lem:Locindec} and Corollary~\ref{cor:Loc-bij}
there is an indecomposable
rigid $\widehat{N} \in \rep_\lf(\hH)$ with
$\Loc(\widehat{N})\cong N$. 
We know that
$\rkv_\hH(\widehat{N}) = \dimv_\tH(N)$.
Now, $\Red(\widehat{N}) \in \rep_\lf(H)$ is indecomposable rigid by
Lemma~\ref{lem:rig-indec}.
We obtain
\[
\rkv_H(\Red(\widehat{N})) = \rkv_\hH(\widehat{N}) = \alp.
\]

This finishes the proof of 
Theorem~\ref{thm:mainresult2}(a).

\subsection{Bongartz complement for hereditary algebras}
\label{subsec:Bongartz}
For the proof of Theorem~\ref{thm:mainresult2}(b),
we need the following result about hereditary algebras, which should be well known.

\begin{Lem} \label{lem:bong} 
Let $A$ be a finite-dimensional hereditary $\FF$-algebra, and let $T\in\rep(A)$ be indecomposable rigid and non-projective.  
Then there exists $X\in\rep(A)$ such that $T\oplus X$
  is a tilting module and $\Hom_A(T,X)=0$.
\end{Lem}

\begin{proof} 
By \cite{Rin94}, $B:=\End_A(T)$ is isomorphic to the endomorphism ring of a simple $A$-module. In particular,
  $B$ is a division algebra. 
Let $s$ be the dimension of the (right) $B$-vector
  space $\Ext_A^1(T,A)$, and let
  $\eta_1,\ldots,\eta_s$ be a $B$-basis of $\Ext_A^1(T,A)$.
Following the proof of \cite[Lemma~2.1]{B81} let
\[
0 \to A \to X \to T^s \to 0
\]
be the short exact sequence whose pullback under the $i$th inclusion
$T \to T^s$ is $\eta_i$ for $1 \le i \le s$.
Applying $\Hom_A(T,-)$ yields an exact sequence
\begin{multline*}
0 \to \Hom_A(T,A) \to \Hom_A(T,X) \to \Hom_A(T,T^s)\\
\xrightarrow{\delta} \Ext_A^1(T,A) \to \Ext_A^1(T,X)\to 0
\end{multline*}
of $B$-linear maps,
where the connecting homomorphism $\delta$ is surjective by construction.
Since $T$ is a brick, $\delta$ is for dimension reasons an isomorphism.
Thus 
\[
\Hom_A(T,A) \cong \Hom_A(T,X) \quad\text{and}\quad \Ext^1_A(T,X)=0.
\]
It is easy to see that also $\Ext^1_A(X,X)=0=\Ext^1_A(X,T)$. Thus $T\oplus X$
is by definiton a tilting module. Finally, since $A$ is hereditary and
$T$ is indecomposable and non-projective, we have $\Hom_A(T,A)=0$.
\end{proof}

\subsection{{Proof of Theorem~\ref{thm:mainresult2}(b)}}\label{subsec:proof3}
We begin with the following key result, which is interesting on its own right.

\begin{Prop} \label{prop:hHendo}
Let $M\in\rep_\lf(\hH)$ be indecomposable rigid. 
Then
\bit

\item[(a)]
$\bil{\rkv_\hH(M),\rkv_\hH(M)} = c_i$ for some $i\in I$.

\item[(b)]
$\End_\hH(M) \cong \FF\pow{\eta}$ as an $R$-algebra,
where $\eta$ is a variable.

\item[(c)]
$M$ is free as an $\End_\hH(M)$-module.

\eit
\end{Prop}

\begin{proof}
Recall that $\Loc(M)\in\rep(\tH)$ is indecomposable rigid by
Lemma~\ref{lem:Locindec} and 
Proposition~\ref{prop:Loc-inj}.
Thus $\rkv_\hH(M)=\dimv_\tH(\Loc(M))$ is a real
Schur root. This shows (a).

In order to show (b) let us consider first the case when $M=\hH e_i$ is
an indecomposable  projective  $\hH$-module. 
In this case, we have
\[
\End_\hH(\hH e_i)\cong e_i\hH e_i\cong\FF\pow{\veps_i}.
\]

Next, let $M$ be non-projective.
As a consequence of Corollary~\ref{cor:Loc-bij}, 
$\Loc(M)\in\rep(\tH)$ is indecomposable rigid and non-projective. 
Let $N \in \rep_\lf(\hH)$ be rigid such that
$\Loc(N)$ is the Bongartz complement of $\Loc(M)$. 
In particular, $M \oplus N$ is a tilting $\hH$-module. 
(Here we used Proposition~\ref{prop:EcT-iso} and Corollary~\ref{cor:Loc-bij}.)
By Lemma~\ref{lem:bong}  we have 
$\Hom_\tH(\Loc(M),\Loc(N)) = 0$.
Since $\Loc$ is faithful by Corollary~\ref{cor:faithful},
this implies $\Hom_\hH(M,N)=0$.
We conclude that
\[
\gldim(\End_\hH(M)) \leq \gldim(\End_\hH(M\oplus N)) \leq \gldim(\hH)+1 \leq 3.
\]
Here the first inequality holds because
$\Hom_\hH(M,N)=0$.
The second inequality holds by \cite[Proposition~3.6.1]{CF04}, since
$M \oplus N \in \rep_\lf(\hH)$ is a tilting module. 
The last inequality holds by Proposition~\ref{prop:hH-pdim}(c). 
  
On the other hand,  $E := \End_\hH(M)$ is a local 
$R$-algebra, which is free of rank $c_i$ as an $R$-module,
see Lemma~\ref{lem:homfree}.
By Lemma~\ref{lem:functors3} we have
\[
E_\veps \cong K \otimes_R E \cong \End_\tH(\Loc(M)) \cong \FF\lpw{\veps_i}
\]
with $c_i = \bil{\rkv_\hH(M),\rk_\hH(M)}$.
In particular, $E$ is a  commutative integral domain.
Thus $E$ is a commutative complete local
$R$-algebra of Krull dimension $1$ with maximal ideal
$\mathfrak{m}$. 
Since $E/(\veps E)$ is a finite-dimensional local $\FF$-algebra, we conclude that the residue field 
\[
\GG := E/\mathfrak{m}
\] 
is a finite extension of $\FF$, say of degree $[\GG:\FF] = d$.
Since moreover $E$ has finite global dimension, we conclude that $E$ is regular. 
Thus there is an isomorphism 
\[
\phi\df E \to \GG\pow{\eta}
\]
of $\FF$-algebras , where $\eta$ is a variable,
see for example \cite[Theorem~19.12, Proposition~10.16]{Eis95}. 

Claim 1: $\GG = \FF$.

Proof: 
Let $\iota\df R \to E$ be the obvious embedding.
This turns $E$ into an $R$-algebra.
The composition
\[
R \xrightarrow{\iota} E \xrightarrow{\phi} \GG\pow{\eta}
\]
is $\FF$-linear, and it gives an embedding of $R$
into $\GG\pow{\eta}$.
In this way $\GG\pow{\eta}$ can be seen as an $R$-algebra,
and $\phi$ becomes an $R$-algebra isomorphism.

By abuse of notation we write $\veps$ for $\iota(\veps)$ and
for $(\phi \circ \iota)(\veps)$.
Since $\veps$ is not invertible in $\GG\pow{\eta}$, we have
\[
\veps \in \eta \GG\pow{\eta}.
\]
It follows that
\[
\GG\pow{\eta}_\veps = \GG\pow{\eta}_\eta = \GG\lpw{\eta}.
\]
In particular, we have an $\FF$-algebra isomorphism
\[
\FF\lpw{\veps_i} \to \GG\lpw{\eta}.
\]

Now let $\alpha \in \GG \setminus \FF$, and let 
$p \in \FF[X]$ be the minimum polynomial of $\alpha$
over $\FF$.
Then $p$ is reducible in $\GG[X]$ and therefore also reducible
in $\GG\lpw{\veps}[X]$.

Claim 2:
$p$ is irreducible in $\FF\lpw{\veps_i}[X]$.

Proof:
Let
\[
{\rm ev}_\alpha\df \FF\lpw{\veps_i}[X] \to \FF(\alpha)\lpw{\veps_i}
\]
be the $\FF\lpw{\veps_i}$-algebra homomorphism, which is defined
by $X \mapsto \alpha$.
Let $m := \deg(p)$.
Then $\FF\lpw{\veps_i} \subset \FF(\alpha)\lpw{\veps_i}$ is
a field extension of degree $m$.
Furthermore $(p) \subseteq \Ker({\rm ev}_\alpha)$.
Since ${\rm ev}_\alpha$ is surjective, we get 
$(p) = \Ker({\rm ev}_\alpha)$ for dimension reasons.
Thus $p$ must be irreducible in $\FF\lpw{\veps_i}[X]$.
This proves Claim 2.

We established that $p$ is irreducible over $\FF\lpw{\veps_i}$
and reducible over $\GG\lpw{\eta}$.
This is clearly a contradiction to the existence of the $\FF$-algebra isomorphism
\[
\FF\lpw{\veps_i} \to \GG\lpw{\eta}.
\]
It follows that $\FF = \GG$, which finally proves our Claim 1.

Thus we know that $\phi$ is after all an $R$-algebra isomorphism
\[
\phi\df E \to \FF\pow{\eta}.
\]
This finishes the proof of (b).

To show (c), we use that $E$ is a principal ideal domain,
which was established in (b). 
We also know that $M$ is free and finitely generated 
as an $R$-module.
Now, suppose that $M$ is not free as an $E$-module, then 
\[
{_E}M = {_E}M' \oplus {_E}M''
\]
with $M'$ free and $M'' \neq 0$ of finite length. 
By restricting the action of $E$ to $R$, the above decomposition
yields a direct sum decomposition
\[
{_R}M = {_R}M' \oplus {_R}M''
\]
with ${_R}M''$ non-zero and of finite length.
This contradiction shows (c).
\end{proof}

\begin{Cor}\label{cor:Hendo}
Let $M\in\rep_\lf(H)$ be indecomposable rigid. 
Then
\bit

\item[(a)]
$\End_H(M) \cong \FF[\eps_i]/(\eps_i^{c_i})$ with
$c_i = \bil{\rkv_H(M),\rkv_H(M)}$.

\item[(b)]
$M$ is free as an $\End_H(M)$-module.

\eit
\end{Cor}

\begin{proof}
(a):
By Proposition~\ref{prop:Red-bij}
there is some indecomposable rigid $\hH$-module 
$\widehat{M} \in \rep_\lf(\hH)$ with 
\[
\Red(\widehat{M}) \cong M.
\]
Let $\widehat{E} := \End_\hH(\widehat{M})$.
By Lemma~\ref{lem:homfree}, $\widehat{E}$ is free as an $R$-module.
It follows from Lemmas~\ref{lem:rig-indec} and \ref{lem:hH-unique} that the rank of this free
$R$-module is $c_i = \bil{\rkv_\hH(\widehat{M}),\rkv_\hH(\widehat{M})}$.

This implies
\[
\dim_\FF \widehat{E}/(\veps \widehat{E}) = c_i.
\]
Since $\widehat{E}$ is isomorphic to a power series algebra by Proposition~\ref{prop:hHendo}(b), this implies
\[
\widehat{E}/(\veps \widehat{E}) \cong \FF[\eps_i]/(\eps_i^{c_i}).
\]
Set $E := \End_H(M)$.
By Corollary~\ref{cor:hH-unique} there is an isomorphism
\[
E \cong \widehat{E}/(\veps \widehat{E}).
\]
This finishes the proof of (a).

(b):
Since $\widehat{M}$ is free as an $\widehat{E}$-module by Proposition~\ref{prop:hHendo}(c), we get that
$M \cong \widehat{M}/(\veps \widehat{M})$ is free as a module
over $E \cong \widehat{E}/(\veps \widehat{E})$.
\end{proof}

Clearly, Corollary~\ref{cor:Hendo} implies  Theorem~\ref{thm:mainresult2}(b).

\subsection{{Proof of Theorem~\ref{thm:mainresult2}(c)}}\label{subsec:pf-c}
Let $M\in\rep_\lf(H)$ be indecomposable rigid, and let
 $E := \End_H(M)$.
By Corollary~\ref{cor:Hendo}(a) we have
\[
E \cong \FF[\veps]/(\veps^{c_i})
\] 
where $c_i = \bil{\rkv_H(M),\rkv_H(M)}$.
Furthermore, by Corollary~\ref{cor:Hendo}(b)
$M$ is free as an $E$-module.
This implies that the $j$th component $d_j$ of the dimension vector $\dimv(M/\rad_E(M))$ is related to the $j$th component $r_j$ of the rank vector $\rkv_H(M)$
by
\[
d_j = \frac{c_j}{c_i} r_j. 
\] 
We know that 
\[
\beta := \sum_j r_j\alpha_j \in \Delta_\rS(C,\Omega).
\]
Let $\widetilde{\beta}$ and $\widetilde{\alpha}_j$ denote the 
dual roots for $(C^T,\Omega)$, as in Section~\ref{sec:schurroots}.
We have 
\[
 \widetilde{\beta} = \frac{c\beta}{\langle\beta,\beta\rangle} = \frac{c\beta}{c_i},
\]
and similarly
\[
\widetilde{\alpha}_j = \frac{c\alpha_j}{c_j}, 
\]
hence
\[
\sum_j d_j \widetilde{\alpha}_j = \sum_j \frac{c_j}{c_i} r_j \frac{c\alpha_j}{c_j} = \frac{c}{c_i}\sum_jr_j\alpha_j=\widetilde{\beta}.
\]
Therefore $\dimv(M/\rad_E(M))$ is the dual of the Schur root $\rkv_H(M)$ expressed in the basis $(\widetilde{\alpha}_j)$.

By Demonet's Lemma~\ref{lem:demonet}
the $\tau$-rigid
$H$-modules are precisely the rigid locally free $H$-modules. 
As a consequence,
by the DIJ-correspondence \cite[Theorem~4.1]{DIJ17}
\begin{align*}
\{ \text{indec. $\tau$-rigids in $\rep(H)$} \}/\!\!\cong \;\;
&\longrightarrow \;\;
\{ \text{left finite bricks in $\rep(H)$} \}/\!\!\cong
\\
M &\mapsto M/\rad_E(M)
\end{align*}
the $H$-module $M/\rad_E(M)$ is a left finite brick, and all
left finite bricks are of this form.
For the definition of a left finite brick we refer to 
\cite[Section~1]{As18}.
Note that $M/\rad_E(M)$ is in general not locally free. 
This concludes the proof of  Theorem~\ref{thm:mainresult2}(c).


\section{Examples}\label{sec:examples}


\subsection{Type $B_3$}
The following example is discussed in \cite[Section~13.7]{GLS1}.
Let
\[
C =
\begin{pmatrix}
2 & -1 & 0\\
-1 & 2 & -1\\
0 & -2 & 2
\end{pmatrix}
\]
be a Cartan matrix of type $B_3$
with symmetrizer $D = \diag(2,2,1)$
and orientation $\Omega = \{ (1,2),(2,3) \}$.
The algebra $H = H_\FF(C,D,\Omega)$ is then
given by the quiver
\[
\xymatrix{
1 \ar@(ul,ur)^{\eps_1} & 2 \ar[l]^{\alpha_{12}}\ar@(ul,ur)^{\eps_2} & 3
\ar[l]^{\alpha_{23}}
}
\]
with relations $\eps_1^2 = \eps_2^2 = 0$ and
$\eps_1\alpha_{12} = \alpha_{12}\eps_2$.
The Auslander-Reiten quiver of 
$H$ is shown in Figure~\ref{Fig:B3}.
As vertices we have the graded dimension vectors (arising from the obvious $\Z$-covering of $H$) of the indecomposable 
$H$-modules. 
(The three modules on the leftmost column have to be identified
with the corresponding three modules on the rightmost column.)
The indecomposable rigid locally free $H$-modules are framed. 
The corresponding left finite bricks are colored in blue.
Thus the real Schur roots in $\Delta_\rS(C,\Omega)$ are
\begin{align*}
(1,1,1), && (0,1,1), && (0,0,1),
\\
(1,1,0), && (1,2,2), && (0,1,2),
\\
(1,0,0), && (0,1,0), && (1,1,2),
\end{align*}
and the corresponding real Schur roots in $\Delta_\rS(C^T,\Omega)$ are
\begin{align*}
(2,2,1), && (0,2,1), && (0,0,1),
\\
(1,1,0), && (1,2,1), && (0,1,1),
\\
(1,0,0), && (0,1,0), && (1,1,1).
\end{align*}
\begin{landscape}
{\tiny
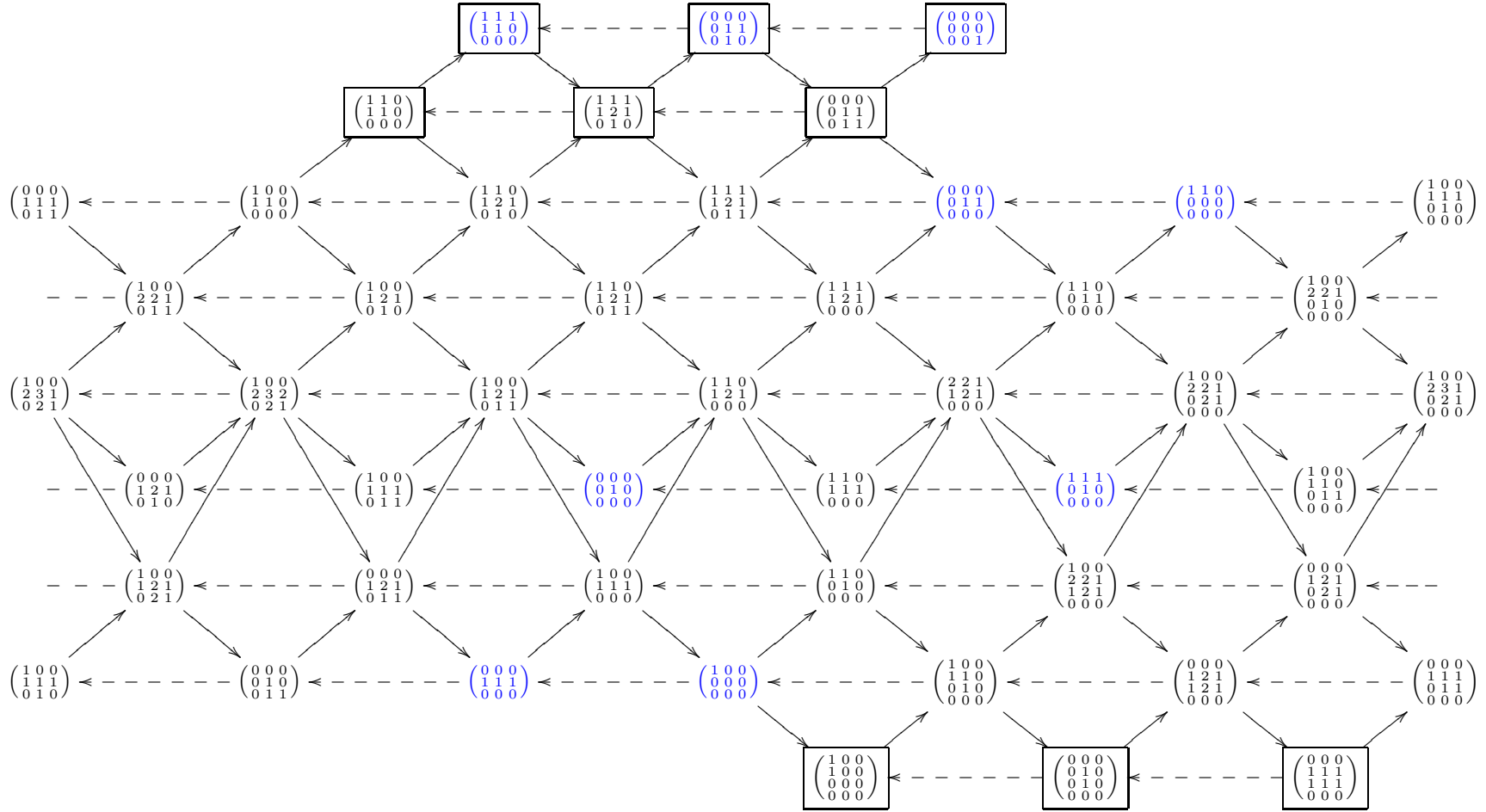
\begin{figure}[!htb]
\[
\xymatrix@-2.5ex{
&&&&*+[F]{\blue{\bsm 1&1&1\\1&1&0\\0&0&0 \esm}}\ar[dr]
&&*+[F]{\blue{\bsm 0&0&0\\0&1&1\\0&1&0 \esm}}\ar[dr]\ar@{-->}[ll]
&&*+[F]{\blue{\bsm 0&0&0\\0&0&0\\0&0&1 \esm}}\ar@{-->}[ll]
&&&&
\\
&&&*+[F]{\bsm 1&1&0\\1&1&0\\0&0&0 \esm}\ar[ur]\ar[dr]
&&*+[F]{\bsm 1&1&1\\1&2&1\\0&1&0 \esm}\ar[ur]\ar[dr]\ar@{-->}[ll]
&&*+[F]{\bsm 0&0&0\\0&1&1\\0&1&1 \esm}\ar@{-->}[ll]\ar[ur]\ar[dr]
\\
{\bsm 0&0&0\\1&1&1\\0&1&1 \esm}\ar[dr]
&&{\bsm 1&0&0\\1&1&0\\0&0&0 \esm}\ar@{-->}[ll]\ar[dr]\ar[ur]
&&{\bsm 1&1&0\\1&2&1\\0&1&0 \esm}\ar@{-->}[ll]\ar[dr]\ar[ur]
&&{\bsm 1&1&1\\1&2&1\\0&1&1 \esm}\ar@{-->}[ll]\ar[dr]\ar[ur]
&&{\blue{\bsm 0&0&0\\0&1&1\\0&0&0 \esm}}\ar@{-->}[ll]\ar[dr]
&&{\blue{\bsm 1&1&0\\0&0&0\\0&0&0 \esm}}\ar@{-->}[ll]\ar[dr]
&&{\bsm 1&0&0\\1&1&1\\0&1&0\\0&0&0 \esm}\ar@{-->}[ll]
\\
&{\bsm 1&0&0\\2&2&1\\0&1&1 \esm}\ar[ur]\ar[dr]\ar@{--}[l]
&&{\bsm 1&0&0\\1&2&1\\0&1&0 \esm}\ar[dr]\ar[ur]\ar@{-->}[ll]
&&{\bsm 1&1&0\\1&2&1\\0&1&1 \esm}\ar[ur]\ar[dr]\ar@{-->}[ll]
&&{\bsm 1&1&1\\1&2&1\\0&0&0 \esm}\ar[ur]\ar[dr]\ar@{-->}[ll]
&&{\bsm 1&1&0\\0&1&1\\0&0&0 \esm}\ar[ur]\ar[dr]\ar@{-->}[ll]
&&{\bsm 1&0&0\\2&2&1\\0&1&0\\0&0&0 \esm}\ar[ur]\ar[dr]
\ar@{-->}[ll]&\ar@{-->}[l]
\\
{\bsm 1&0&0\\2&3&1\\0&2&1 \esm}\ar[dr]\ar[ur]\ar[ddr]
&&{\bsm 1&0&0\\2&3&2\\0&2&1 \esm}\ar[dr]\ar[ur]\ar[ddr]\ar@{-->}[ll]
&&{\bsm 1&0&0\\1&2&1\\0&1&1 \esm}\ar[dr]\ar[ur]\ar[ddr]\ar@{-->}[ll]
&&{\bsm 1&1&0\\1&2&1\\0&0&0 \esm}\ar[dr]\ar[ur]\ar[ddr]\ar@{-->}[ll]
&&{\bsm 2&2&1\\1&2&1\\0&0&0 \esm}\ar[dr]\ar[ur]\ar[ddr]\ar@{-->}[ll]
&&{\bsm 1&0&0\\2&2&1\\0&2&1\\0&0&0 \esm} \ar[dr]\ar[ur]\ar[ddr]\ar@{-->}[ll]
&&{\bsm 1&0&0\\2&3&1\\0&2&1\\0&0&0 \esm}\ar@{-->}[ll]
\\
&{\bsm 0&0&0\\1&2&1\\0&1&0 \esm}\ar[ur]\ar@{--}[l]
&&{\bsm 1&0&0\\1&1&1\\0&1&1 \esm}\ar[ur]\ar@{-->}[ll]
&&{\blue{\bsm 0&0&0\\0&1&0\\0&0&0 \esm}}\ar[ur]\ar@{-->}[ll]
&&{\bsm 1&1&0\\1&1&1\\0&0&0 \esm}\ar[ur]\ar@{-->}[ll]
&&{\blue{\bsm 1&1&1\\0&1&0\\0&0&0 \esm}}\ar[ur]\ar@{-->}[ll]
&&{\bsm 1&0&0\\1&1&0\\0&1&1\\0&0&0 \esm}\ar[ur]\ar@{-->}[ll]
&\ar@{-->}[l]
\\
&{\bsm 1&0&0\\1&2&1\\0&2&1 \esm}\ar[uur]\ar[dr]\ar@{--}[l]
&&{\bsm 0&0&0\\1&2&1\\0&1&1 \esm}\ar[uur]\ar[dr]\ar@{-->}[ll]
&&{\bsm 1&0&0\\1&1&1\\0&0&0 \esm}\ar[uur]\ar[dr]\ar@{-->}[ll]
&&{\bsm 1&1&0\\0&1&0\\0&0&0 \esm}\ar[uur]\ar[dr]\ar@{-->}[ll]
&&{\bsm 1&0&0\\2&2&1\\1&2&1\\0&0&0 \esm}\ar[uur]\ar[dr]
\ar@{-->}[ll]
&&{\bsm 0&0&0\\1&2&1\\0&2&1\\0&0&0 \esm}\ar[uur]\ar[dr]
\ar@{-->}[ll]&\ar@{-->}[l]
\\
{\bsm 1&0&0\\1&1&1\\0&1&0 \esm}\ar[ur]
&&{\bsm 0&0&0\\0&1&0\\0&1&1 \esm}\ar[ur]\ar@{-->}[ll]
&&{\blue{\bsm 0&0&0\\1&1&1\\0&0&0 \esm}}\ar[ur]\ar@{-->}[ll]
&&{\blue{\bsm 1&0&0\\0&0&0\\0&0&0 \esm}}\ar[ur]\ar[dr]\ar@{-->}[ll]
&&{\bsm 1&0&0\\1&1&0\\0&1&0\\0&0&0 \esm}\ar[ur]\ar[dr]
\ar@{-->}[ll]
&&{\bsm 0&0&0\\1&2&1\\1&2&1\\0&0&0 \esm}\ar[ur]\ar[dr]
\ar@{-->}[ll]
&&{\bsm 0&0&0\\1&1&1\\0&1&1\\0&0&0 \esm}\ar@{-->}[ll]
\\
&&&&&&&
*+[F]{\bsm 1&0&0\\1&0&0\\0&0&0\\0&0&0 \esm}\ar[ur]
&&*+[F]{\bsm 0&0&0\\0&1&0\\0&1&0\\0&0&0 \esm}\ar[ur]\ar@{-->}[ll]
&&*+[F]{\bsm 0&0&0\\1&1&1\\1&1&1\\0&0&0 \esm}\ar[ur]\ar@{-->}[ll]
&
}
\] 
\caption{
The Auslander-Reiten quiver of $H_\FF(C,D,\Omega)$ of type $B_3$ with
$D$ minimal.}
\label{Fig:B3}
\end{figure}
}
\end{landscape}

\subsection{Type $\widetilde{C}_2$}
Let
\[
C =
\begin{pmatrix}
2 & -1 & 0\\
-2 & 2 & -2\\
0 & -1 & 2
\end{pmatrix}
\]
be a Cartan matrix of type $\widetilde{C}_2$
with symmetrizer $D = \diag(2,1,2)$
and orientation $\Omega = \{ (1,2),(2,3) \}$.
The algebra $H = H_\FF(C,D,\Omega)$ is then
given by the quiver
\[
\xymatrix{
1 \ar@(ul,ur)^{\eps_1} & 2 \ar[l]^{\alpha_{12}} & 3
\ar[l]^{\alpha_{23}}\ar@(ul,ur)^{\eps_3}
}
\]
with relations $\eps_1^2 = \eps_3^2 = 0$.
Thus $H$ is a representation-infinite gentle algebra.
Thus $H$ is a string algebra in the sense of \cite{BR87}.
For each string $C$ let $M(C)$ be the corresponding string module,
see \cite{BR87} for detailed definitions.
For $i = 1,2,3$ let $1_i$ be the string of length $0$ associated with
$i$.
(Then $M(1_i)$ is the simple $H$-module $S_i$.)

Let $P(i)$ (resp. $I(i)$) be the indecomposable projective (resp. injective) $H$-modules associated to the vertex $i \in \{ 1,2,3 \}$.
Up to isomorphism, we have then 
\begin{align*}
P(1) &= M(p_1), &
P(2) &= M(p_2), &
P(3) &= M(p_3),
\\
I(1) &= M(q_1), &
I(2) &= M(q_2), &
I(3) &= M(q_3),
\end{align*}
where 
\begin{align*}
p_1 &:= \eps_1, &
p_2 &:= \eps_1\alpha_{12}, &
p_3 &:= \eps_1\alpha_{12}\alpha_{23}\eps_3\alpha_{23}^{-1}
\alpha_{12}^{-1}\eps_1^{-1},
\\
q_1 &:= \eps_3^{-1}\alpha_{23}^{-1}\alpha_{12}^{-1}\eps_1
\alpha_{12}\alpha_{23}\eps_3, &
q_2 &= \alpha_{23}\eps_3, &
q_3 &= \eps_3.
\end{align*}
We get
\begin{align*}
\rkv_H(P(1)) &= (1,0,0), &
\rkv_H(P(2)) &= (1,1,0), &
\rkv_H(P(3)) &= (2,2,1),
\\
\rkv_H(I(1)) &= (1,2,2), &
\rkv_H(I(2)) &= (0,1,1), &
\rkv_H(I(3)) &= (0,0,1).
\end{align*}

Now let 
\begin{align*}
h_1 &:= \alpha_{12}, &
h_2 &:= \alpha_{23}\eps_3^{-1}\alpha_{23}^{-1}\alpha_{12}^{-1}\eps_1^{-1},
\\
c_2 &:=
\alpha_{12}^{-1}\eps_1\alpha_{12}\alpha_{23}\eps_3,
& 
c_3 &:= \alpha_{23}^{-1}.
\end{align*}
Then $h_1$ and $h_2$ are hooks and $c_2$ and $c_3$ are
cohooks in the sense of \cite{BR87}.
By \cite{BR87} we have
\begin{align*}
\tau^{-2n}(P(2)) &= M((h_1h_2)^{-n}p_2(h_2h_1)^n),
\\
\tau^{-2n-1}(P(2)) &= M(h_1^{-1}(h_1h_2)^{-n}p_2(h_2h_1)^nh_2),
\\
\tau^{-2n}(P(i)) &= M((h_1h_2)^{-n}p_i(h_1h_2)^n),
\\
\tau^{-2n-1}(P(i)) &= M(h_1^{-1}(h_1h_2)^{-n}p_i(h_1h_2)^nh_1),
\\
\tau^{2n}(I(2)) &= M((c_2c_3)^{-n}q_2(c_3c_2)^n),
\\
\tau^{2n+1}(I(2)) &= M(c_2^{-1}(c_2c_3)^{-n}q_2(c_3c_2)^nc_3),
\\
\tau^{2n}(I(i)) &= M((c_3c_2)^{-n}q_i(c_3c_2)^n),
\\
\tau^{2n+1}(I(i)) &= M(c_3^{-1}(c_3c_2)^{-n}q_i(c_3c_2)^nc_3)
\end{align*}
for $i = 1,3$ and $n \ge 0$.
Furthermore, set
\begin{align*}
R_1 &:= M(1_2), &
R_2 &:= M(\eps_1\alpha_{12}\alpha_{23}\eps_3).
\end{align*}
The modules $\tau^{-n}(P(i))$ and $\tau^n(I(i))$
for $i \in \{1,2,3\}$ and $n \ge 0$ are the preprojective resp. 
preinjective $H$-modules as defined in \cite[Section~1.5]{GLS1}.
The modules $R_1$ and $R_2$ form the bottom of a tube of
rank $2$ in the Auslander-Reiten quiver of $H$.

Ricke \cite{R16} showed that the modules
$\tau^{n}(I(i))$ and $\tau^{-n}(P(i))$ with $i \in \{1,2,3\}$ and 
$n\ge 0$ together with
$R_1$ and $R_2$ form a complete set of representatives
of isoclasses of indecomposable rigid $H$-modules, and that
all of these are locally free.

Define
\begin{align*}
p_1' &:= 1_1, &
p_3' &:= \alpha_{23}^{-1}
\alpha_{12}^{-1}\eps_1^{-1},
\\
q_1' &:= \alpha_{12}\alpha_{23}\eps_3, &
q_3' &:= 1_3.
\end{align*}

Under the DIJ-correspondence \cite[Theorem~4.1]{DIJ17}
we get
\begin{align*}
\tau^{-n}(P(2)) &\mapsto \tau^{-n}(P(2)), 
\\
\tau^{-2n}(P(i)) &\mapsto M(p_i'(h_1h_2)^n),
\\
\tau^{-2n-1}(P(i)) &\mapsto M(p_i'(h_1h_2)^nh_1),
\\
\tau^n(I(2)) &\mapsto \tau^n(I(2)),
\\
\tau^{2n}(I(i)) &\mapsto M(q_i'(c_3c_2)^n),
\\
\tau^{2n+1}(I(i)) &\mapsto M(q_i'(c_3c_2)^nc_3),
\\
R_1 &\mapsto R_1,
\\
R_2 &\mapsto R_2
\end{align*}
for $i = 1,3$ and $n \ge 0$.

Now it is easy to compute the rank vectors of the indecomposable
$\tau$-rigids and the corresponding dimension vectors of the 
left finite bricks.
For example, for $n \ge 0$ we have
\[
\alpha(n) :=
\rkv_H(\tau^{-2n-1}(P(3))) = 
(2,2,1) + n(2,4,2) + (0,2,0).
\]
For the
corresponding left finite bricks we get the expected dimension
vector
\[
\widetilde{\alpha(n)} = (2,1,1) + n(2,2,2) + (0,1,0).
\]
(We know from \cite{GLS1} that 
\[
\bil{\rkv_H(\tau^{-n}(P(i))),\rkv_H(\tau^{-n}(P(i)))} = 
\bil{\rkv_H(\tau^n(I(i))),\rkv_H(\tau^n(I(i)))} 
= c_i
\]
for all $i \in \{1,2,3\}$ and $n \ge 0$. 
This makes it easy to calculate
$\widetilde{\alpha(n)}$.)

Note that 
there is a ${\mathbb P}^1(\FF)$-parameter
family of bricks in $\rep(H)$ which are not left finite.
This phenomenon occurs for example also for the Kronecker quiver.

\bigskip
{\parindent0cm \bf Acknowledgements.}\,
The first named author acknowledges partial support from 
CoNaCyT grant no. 239255, and he thanks the Max-Planck Institute for
Mathematics in Bonn for one year of hospitality in 2017/18.
The third author thanks the SFB/Transregio TR 45 for 
financial support. 
We thank Laurent Demonet, Lidia Angeleri H\"ugel and Henning Krause for helpful 
discussions and for providing useful references.


\end{document}